%% file: main.tex
\let\ORIlabel\label
\let\ORIrefstepcounter\refstepcounter
\AddToHook{package/hyperref/before}{
   \let\label\ORIlabel 
   \let\refstepcounter\ORIrefstepcounter}
\documentclass[onefignum,onetabnum,reqno]{siamart251216}
\input{preamble}
\title{A radial and tangential framework for studying transient reactivity in two-dimensional systems}
\author{James Broda$^1$ \and Alanna Haslam-Hyde$^{2,*}$ \and Mary Lou Zeeman$^3$}
\begin{document}
\maketitle
\textit{{\small \centering 
	$^1$Washington and Lee University, Department of Mathematics\\
	$^2$Boston University, Department of Mathematics and Statistics\\
	$^3$Bowdoin College, Department of Mathematics\\
	$^*$Corresponding Author: \href{mailto:alanna.haslam@gordon.edu}{alanna.haslam-hyde@gordon.edu}\\[1em]
}}

\begin{abstract}
    \input{abstract}
\end{abstract}

\input{RadTang-1-Intro}

\input{RadTang-2-Defn}

\input{RadTang-3-Recast}
\input{RadTang-4-Eigen}

\input{RadTang-5-Ortho}
\input{RadTang-6-StandardForms}

\input{RadTang-7-max-amp}
\input{RadTang-8-rotating-escape}
\input{RadTang-9-Acknowledgements}
\bibliographystyle{siamplain}
\bibliography{references}
\end{document}

%% file: preamble.tex
\usepackage{amsfonts}
\usepackage{ulem}
\usepackage{graphicx}
\usepackage{enumitem}
\usepackage{algorithmic}
\usepackage{hyperref, color}
\usepackage{lscape}
\usepackage{changepage}
\usepackage{tabularx,multirow}
\usepackage{nicefrac, xfrac}
\usepackage{lineno}
\usepackage{csquotes}
\usepackage[font=small,skip=6pt]{caption}
\newcommand{\R}{\mathbb{R}}

\newcommand{\tr}{\text{tr}}

\newcommand{\Rad}{\mathcal{R}}
\newcommand{\Tang}{\mathcal{T}}

\newcommand{\Rset}{\mathcal{S}_R}
\newcommand{\Rreg}{\mathcal{U}_R}
\newcommand{\Aset}{\mathcal{S}_C}
\newcommand{\Areg}{\mathcal{U}_C}
\newcommand{\Remark}{\textit{Remark}.~}
\newcommand{\Mod}[1]{\ (\mathrm{mod}\ #1)} 

\newtheorem{theo}{Theorem}[section]
\newtheorem{prop}[theo]{Proposition}
\newtheorem{defn}[theo]{Definition}
\newtheorem{cor}[theo]{Corollary}
\newtheorem{lem}[theo]{Lemma}
\newtheorem{notn}[theo]{Notation}


%% file: abstract.tex
Even if a linear system of ordinary differential equations has a globally attracting equilibrium at the origin, small disturbances from the equilibrium may lead to large transient excursions before the system stabilizes. 
This counter-intuitive phenomenon of transient amplification is called {\it reactivity} and is often associated with systems that are non-normal.
Here, we establish a new framework for analyzing reactivity and transient dynamics in two-dimensional linear ODEs.  
Our work is facilitated by decomposing the corresponding vector field into sinusoidal {\it radial} and {\it tangential} components. 
Using this decomposition, we introduce a structure of {\it orthovectors} and {\it orthovalues} as dual to the eigenstructure.
Since diagonalization masks transient reactivity, we combine the eigenstructure and the orthostructure to propose alternative matrix forms  which capture both transient and asymptotic behavior and which highlight reactivity features more directly.
Leveraging these matrix forms,
we analytically quantify the maximal amplification in globally attracting systems, and we provide new insight into how 
a nonautonomous linear system can be unstable, even when all the frozen-time systems are stable.

%% file: RadTang-1-Intro.tex
\section{Introduction to reactivity}\label{sec:Intro}
It is a surprising, albeit well-known, fact that, even in a linear system of ordinary differential equations with a global attractor at the origin, trajectories may
take an excursion away from the origin before eventually converging there \cite{neubert1997alternatives}. See Figure \ref{fig:bullseye} and the example therein. 
No nonlinearity or degeneracy is needed in the system for this transient amplification of perturbations from the equilibrium to occur. 
\begin{figure}
    \centering
    \includegraphics[width=\textwidth]{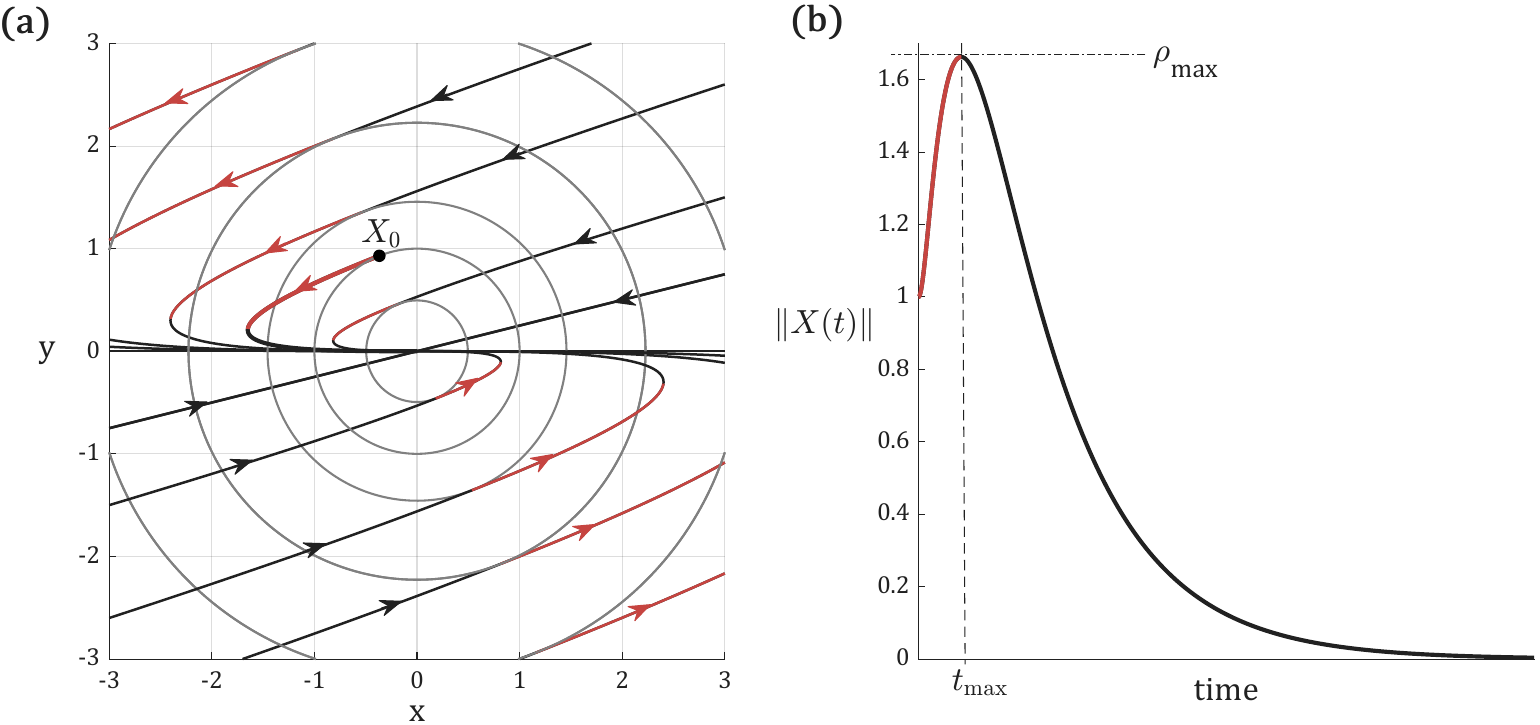}
    \caption{
    \textbf{(a)} The phase portrait of linear ODE \eqref{eqn:ODE-n} with coefficient matrix $\textstyle A=\left(\begin{smallmatrix} -1 & -8\\ 0 & -3\end{smallmatrix}\right)$. 
    The concentric circles indicate contours of equal magnitude. The origin is globally attracting, but many trajectories transiently increase their magnitude (highlighted in red) before converging.
    The solution trajectory, $X(t)$, with initial condition 
    $X_0\approx (-0.4,0.9)$
    on the unit circle is selected to illustrate maximal amplification. \textbf{(b)} The magnitude of $X(t)$ plotted versus time, with the 
    transient amplification highlighted in red. Maximal amplification $\rho_{\max} \approx 1.67$ is attained at $t=t_{\max}$.
    }
    \label{fig:bullseye}
\end{figure}
\par
In their landmark paper, Neubert and Caswell \cite{neubert1997alternatives} coined the term {\it reactivity} to capture the maximal rate of radial amplification of perturbations as follows:
\begin{defn} 
\label{def:reactivity}
Given a real $n \times n$ matrix $A$, consider the linear system
\begin{equation}
\label{eqn:ODE-n}
\frac{dX}{dt}=AX. 
\end{equation}
For any initial condition $X_0 \in \R^n$, let $X(t)$ be the solution of 
the system with $X(0)=X_0$, let $r(t)=\|X(t)\|$, and let $r_0=\|X_0\|$.
The \textup{reactivity} of System \eqref{eqn:ODE-n} is given by 
\begin{equation}
\label{eqn:reactivity}
\rho = \max_{X_0 \neq 0} \left( \frac{1}{r_0} \frac{dr}{dt} \Big\rvert_{t=0} \right)
\end{equation}
and the system is called \textup{reactive} if $\rho>0$.
\end{defn}
Neubert and Caswell \cite{neubert1997alternatives} show that  reactivity is equal to the largest eigenvalue of the symmetric (Hermitian) part of $A$:
\begin{equation}
\label{eqn:reactivity-calc}
\rho = \max(\lambda (H_A))
\end{equation}
where $H_A=\frac{1}{2}(A+A^T)$, 
and $\lambda (H_A)$ denotes the set of eigenvalues (the spectrum) 
of $H_A$. For the case when the origin is an attractor, 
they also introduce the concept of {\it maximal amplification}, $\rho_{\max}$, which measures the extent to which reactivity can amplify perturbations, and
the time, $t_{\max}$, needed to achieve maximal amplification.  See Figure \ref{fig:bullseye}.
\par
In analogy with reactivity, Townley and Hodgson \cite{townley2008erratum} introduce {\it attenuation} as the minimum rate of radial amplification of perturbations. 
System \eqref{eqn:ODE-n} is then called {\it attenuating} if the attenuation is negative, corresponding to (possibly transient) radial decay of solutions. 
\par
Phenomena such as reactivity are often associated with non-normality, and have been extensively analyzed through that lens. See, for example, \cite{higham1993stiffness, trefethen2020spectra, troude2025unifying}.  In this paper we introduce a novel way to calculate the reactivity and attenuation of System \eqref{eqn:ODE-n} when $A$ is $2 \times 2$, by decomposing the vector field $AX$ into its radial and tangential components. The decomposition sheds new light on the beautiful interplay between the eigenstructure and reactivity properties of the system, thereby revealing interesting parallels between its asymptotic and transient behaviors. See also \cite{haslam2020thesis}, where the decomposition was first introduced.
\par \medskip \noindent 
\textbf{Consequences of reactivity.}
The surprising consequences of reactivity are known in many fields of mathematics. 
Neubert and Caswell \cite{neubert1997alternatives} were motivated by the study of resilience of ecological systems to perturbations - a subject that becomes ever more important as climate change and human activity deliver new disturbance patterns to those systems.
\par
Harrington {\it et al.}\ 
\cite{harrington2022reactivity}
consider reactivity and attenuation (in the $\ell_1$ and $\ell_2$ norms), in the context of marine metapopulations. Within their family of models, they show that even in two dimensions the asymptotic and transient dynamics of System \eqref{eqn:ODE-n} can be strikingly different, with populations growing arbitrarily large through reactivity before eventually going extinct. Long transients such as these can present significant challenges for sustainable ecosystem management \cite{hastings2018transient, francis2021management}.
\par
A nice survey of recent reactivity studies across the sciences can be found in \cite{vesipa2017impact}. Reactivity also arises in engineering systems, where methods from control theory are used to shape transient responses of linear and nonlinear systems; see, for example, \cite{duffie2014control} and \cite{nijmeijer1990nonlinear}. 
The reliable production and distribution of electricity, for example, requires engineers to explicitly consider transient dynamics when designing and auditing power systems, since disturbances such as short-circuits or sudden load changes may result in 
blackouts if the system is not sufficiently robust \cite{hatziargyriou2020definition, ruiz2003comprehensive1}. 
\par
In nonautonomous linear systems, transient reactivity can drive asymptotic instability: solutions to a nonautonomous system $X'(t) = B(t)X$ can grow without bound, even when the time-varying coefficient matrix $B(t)$ has eigenvalues with negative real part for every `frozen' time $t$. See \cite{josic2008unstable, lawley2014sensitivity, mierczynski2017instability} and the references therein. The critical mechanism underlying unbounded growth in this setting is the accumulation of transient reactivity from the frozen-time (autonomous) systems. This subtle source of instability can lead to counter-intuitive stability properties of periodic orbits in nonlinear autonomous systems (captured by the Floquet multipliers), and of equilibria in flow-kick systems \cite{meyer2018quantifying}, as both are governed by the nonautonomous variational equation \cite{guckenheimer1983nonlinear, hirsch2012differential}.
\par
When teaching introductory ordinary differential equations (ODEs) or dynamical systems, we often focus on eigenvalue analysis and diagonalization of linear or locally linearized systems to determine the behavior near an equilibrium. There is good reason for this; eigenvalues capture the asymptotic stability of the equilibrium, and solutions of a linear system and its diagonalization are topologically conjugate. But information about the transient behavior of solutions is lost in this process. For example, Higham and Trefethan \cite{higham1993stiffness} point out that if $A$ is diagonal with negative eigenvalues, then System \eqref{eqn:ODE-n} cannot be reactive. They conclude: 
\begin{displayquote} We wish to argue that diagonalization\ldots may change the nature of an ODE significantly, and in fact, that some effects which are customarily attributed to linearization or freezing of coefficients can with greater justice be blamed on diagonalization.
\end{displayquote}
It is our hope that material in this paper will lend itself to inclusion in the curriculum at several levels, to help build students' intuition for reactivity and transient dynamics.
\par \medskip \noindent 
\textbf{Overview.}
Our radial and tangential framework is built on four cornerstones. The first is Theorem \ref{theo:RadTang}, in which we use polar coordinates, $(r,\theta)$, to decompose a two-dimensional linear system (System \eqref{eqn:ODE}) into its radial and tangential components. These components yield $r$-independent sinusoidal functions $\Rad(\theta)$ and $\Tang(\theta)$ that we call the radial and tangential functions, respectively (Definition \ref{def:RT}).
The second cornerstone of the paper, Corollary \ref{cor:RadTangS^1}, links these radial and tangential functions to the dynamics of solutions of the system.
In Section \ref{sec:recast}, we leverage the sinusoidal structure of the radial function, $\Rad$, to reveal more of the structure of reactivity. The third cornerstone of the paper is Definition \ref{def:ReactiveRegion}, where we 
expand on the ideas of reactivity and attenuation to decompose the phase plane of the system into reactive and attenuating regions.
In Section \ref{sec:eigen}, we view the eigenstructure of the system through the lens of $\Rad$ and $\Tang$, 
paving the way for the fourth cornerstone of the paper: a dual structure of orthovectors and orthovalues that quantify
properties of the reactivity and the reactive region (Section \ref{sec:ortho}).
\par
In Sections \ref{sec:SF}--\ref{sec:nonaut}, we build on the four cornerstones. 
In response to the observation that diagonalization masks any transient reactivity or attenuation of a system \cite{higham1993stiffness},  we propose, in Theorem \ref{theo:std_forms}, four standard matrix forms that are easy to calculate and that illuminate the reactivity dynamics. 
The standard forms result from conjugation by pure rotation, and hence respect both the transient and asymptotic behavior of a system. 
In Section \ref{sec:BoundingReactivity}, we quantify the maximal amplification in a system. For reactive systems with an attracting equilibrium, we show  that while the reactivity can be arbitrarily large for any (non-orthogonal) pair of real eigenvectors (Theorem \ref{theo:GivenEigenvectors-AnyReactivity}), the
maximal amplification is bounded (Theorem \ref{theo:MaxAmpUpperBound}). We also describe how to calculate the maximal amplification exactly, with a formula that interweaves the eigen- and ortho-structures of the system (Theorem \ref{theo:MaxAmpFormula}).
In Section \ref{sec:nonaut}, we use our reactivity framework to shed new light on established examples of nonautonomous linear systems in which the origin is unstable, despite being asymptotically attracting for all frozen moments in time.

%% file: RadTang-2-Defn.tex
\section{Accessing radial dynamics directly}
\label{sec:RadTang}
We are interested in reactivity and the transient excursions of trajectories away from the origin. To capture these radial dynamics directly, we write the Cartesian coordinates of $X \in \R^2$ in terms of the polar coordinates $(r, \theta)$, where $r=\|X\| \geq 0$, and $\theta \in \R \Mod{2\pi}$ measures the counter-clockwise angle of $X$ from the $x$-axis. Thus, throughout the paper, we work with the two-dimensional autonomous linear system of ODEs:
\begin{equation}
\label{eqn:ODE}
\frac{dX}{dt}=AX, \mbox{ where } X=\begin{pmatrix} x_1 \\ x_2 \end{pmatrix} 
= \begin{pmatrix} r\cos\theta \\ r\sin\theta \end{pmatrix} \in \R^2 \mbox{ and } A = \begin{pmatrix} a_{11} & a_{12} \\ a_{21} & a_{22}\end{pmatrix}.
\end{equation}
\noindent 
\textbf{Radial and tangential components of $AX$.} 
The first cornerstone of the paper is Theorem \ref{theo:RadTang}, below, which teases apart the radial and tangential components of the vector field $AX$.
Here, by `tangential component', we mean the component of $AX$ tangential to the circle of radius $\|X\|$, and therefore orthogonal to the purely radial component.
See Figure \ref{fig:RTdef}. This straightforward decomposition of $AX$ reveals a surprisingly rich structure with which to analyze reactivity in two-dimensional linear systems. 
First, we introduce some notation that will be used throughout the paper.
\begin{notn} 
\label{notn:J}
Let $J$ and $J^{-1}$ be the matrices representing counter-clockwise and clockwise rotation by $\pi/2$ radians in $\R^2$, respectively. 
For $X = \begin{pmatrix} x_1 \\ x_2 \end{pmatrix} \in \R^2$, let $X^\perp$ denote the vector orthogonal to $X$ obtained by counter-clockwise rotation by $\pi/2$:
\begin{equation*}
J = \begin{pmatrix} 0 & -1 \\ 1 & 0\end{pmatrix}, \ \ J^{-1} = -J = \begin{pmatrix} 0 & 1 \\ -1 & 0\end{pmatrix}, \mbox{ and }
X^\perp = JX = \begin{pmatrix} -x_2 \\ x_1 \end{pmatrix}.
\end{equation*}
\end{notn}
\begin{figure}
    \centering
    \setlength{\abovecaptionskip}{-8pt}
\includegraphics[width=\textwidth]{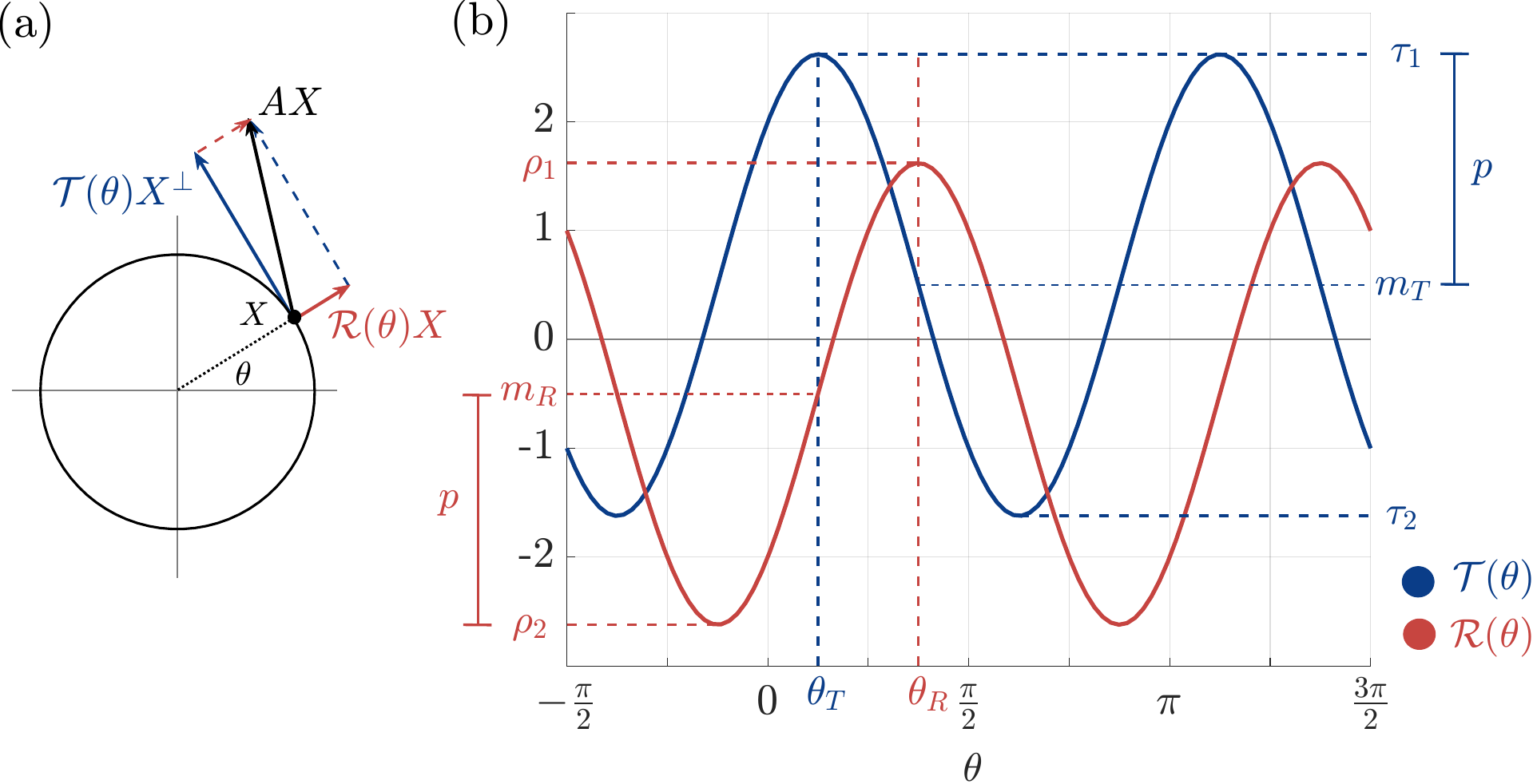}
\setlength{\belowcaptionskip}{-2pt}
    \caption{
    \textbf{(a)} The vector field $AX$ (black) is decomposed into components parallel to (red) and orthogonal to (dark blue) the vector $X$. 
    \textbf{(b)} Graphs of the radial function, $\Rad(\theta)$, in red, and the tangential function, $\Tang(\theta)$, in dark blue, vs. $\theta$ for System \eqref{eqn:ODE} with coefficient matrix $A=\left(\begin{smallmatrix}-2&1\\2&1 \end{smallmatrix}\right)$.
    $\Rad$ and $\Tang$ are sinusoidal with period $\pi$ and amplitude $p$, and are $\pi/4$ out of phase. 
    $\Rad$ has maximum and minimum values $\rho_1$ and $\rho_2$, respectively, midline $m_R$, and maximum at $\theta_R$.
    $\Tang$ has maximum and minimum values $\tau_1$ and $\tau_2$, respectively, midline $m_T$, and maximum at $\theta_T=\theta_R-\pi/4$ (Corollary \ref{cor:sinusoidal}). See
    Table \ref{tab:notation} for links to interactive Desmos pages graphing $\Rad$ and $\Tang$.} 
    \label{fig:RTdef}
\end{figure}
\begin{theo}
\label{theo:RadTang}
Given real $A = \begin{pmatrix} a_{11} & a_{12} \\ a_{21} & a_{22}\end{pmatrix}$ and 
$X=\begin{pmatrix} r\cos\theta \\ r\sin\theta \end{pmatrix}$, 
$r \geq 0, \ \theta \in \R\Mod{2\pi}$, the vector $AX \in \R^2$ can be decomposed into components parallel and orthogonal to $X$ by:
\begin{equation}
\label{eqn:decompose}
AX=\Rad(\theta)X+\Tang(\theta)X^\perp, 
\end{equation}
where $\Rad(\theta)$ and $\Tang(\theta)$ are real-valued functions given by
\begin{align}
\label{eqn:Rad}
\Rad(\theta) & =m_R + p \cos(2(\theta - \theta_R))\\
\label{eqn:Tang}
\Tang(\theta) & =m_T - p \sin(2(\theta - \theta_R))
\end{align}
with 
\begin{align}
\label{eqn:m_R}
m_R & = \frac{1}{2}(a_{11}+a_{22}), \\
\label{eqn:m_T}
m_T & = \frac{1}{2}(a_{21}-a_{12}), \\
\label{eqn:p}
p & = \frac{1}{2} \sqrt{(a_{11}-a_{22})^2 + (a_{12}+a_{21})^2},
\end{align}
and, if $p \neq 0$,
\begin{align}
\label{eqn:theta_R}
\theta_R & = \frac{1}{2}\arctan(a_{12}+a_{21},a_{11}-a_{22})\in \R\Mod{\pi}.
\end{align}
\end{theo}
\par 
\Remark
We use the two-argument form of $\arctan(y,x)\in\R\Mod{2\pi}$ to ensure that $\theta_R$ is well-defined in $\R\Mod{\pi}$, since $\Rad(\theta)$ and $\Tang(\theta)$ are $\pi$-periodic functions.
For the reader's convenience, terms defined throughout the paper are summarized in Table \ref{tab:notation}.
\par 
\begin{proof}[Proof of Theorem \ref{theo:RadTang}]
Projecting $AX$ onto the $X$ and $X^\perp$ directions, we have
\begin{align*}
AX & = \left(AX \cdot \frac{X}{\|X\|}\right)\frac{X}{\|X\|} 
+ \left(AX \cdot \frac{X^\perp}{\|X^\perp\|}\right)\frac{X^\perp}{\|X^\perp\|} \nonumber \\
 & = \left(\frac{AX \cdot X}{\|X\|^2}\right) X + \left(\frac{AX \cdot X^\perp}{\|X^\perp\|^2}\right) X^\perp.
\end{align*} 
Thus, the coefficient of $X$ in Equation \eqref{eqn:decompose} is given by
\begin{align*}
\left(\frac{AX \cdot X}{\|X\|^2}\right) & = \frac{1}{r^2}
\begin{pmatrix} a_{11} & a_{12} \\ a_{21} & a_{22}\end{pmatrix}
\begin{pmatrix} r\cos\theta \\ r\sin\theta \end{pmatrix} \cdot
\begin{pmatrix} r\cos\theta \\ r\sin\theta \end{pmatrix}\\
& = a_{11}\cos^2\theta + (a_{12}+a_{21})\cos\theta\sin\theta + a_{22}\sin^2\theta.
\end{align*} 
Invoking double-angle identities, we have
\begin{align*}
\left(\frac{AX \cdot X}{\|X\|^2}\right) 
& = \frac{1}{2}\big(a_{11}(\cos(2\theta)+1) + (a_{12}+a_{21})\sin(2\theta) + a_{22}(1-\cos(2\theta))\big)\\
& = \frac{1}{2}\big((a_{11}+a_{22})+(a_{11}-a_{22})\cos(2\theta)+(a_{12}+a_{21})\sin(2\theta)\big).
\end{align*}
Rewriting this as a single sinusoid yields
\begin{align*}
\left(\frac{AX \cdot X}{\|X\|^2}\right) 
& = m_R + p \cos(2(\theta - \theta_R)) = \Rad(\theta),
\end{align*}
where $p$ and $\theta_R$ are given by Equations \ref{eqn:p} and \ref{eqn:theta_R}.
Similarly, $\|X^\perp\|=\|X\|=r$, and
the coefficient of $X^\perp$ in Equation \eqref{eqn:decompose} is given by
\begin{align*}
\left(\frac{AX \cdot X^\perp}{\|X^\perp\|^2}\right) & = \frac{1}{r^2}
\begin{pmatrix} a_{11} & a_{12} \\ a_{21} & a_{22}\end{pmatrix}
\begin{pmatrix} r\cos\theta \\ r\sin\theta \end{pmatrix} \cdot
\begin{pmatrix} -r\sin\theta \\ r\cos\theta \end{pmatrix}\\
& = (a_{22}-a_{11})\cos\theta\sin\theta - a_{12}\sin^2\theta + a_{21}\cos^2\theta \\
& = \frac{1}{2}((a_{21}-a_{12})+(a_{12}+a_{21})\cos(2\theta)+(a_{22}-a_{11})\sin(2\theta)) \\
& = m_T - p \sin(2(\theta - \theta_R) \\
& =  \Tang(\theta). 
\end{align*} 
Thus, $AX=\Rad(\theta)X+\Tang(\theta)X^\perp$. 
\end{proof}
\par 
\begin{defn}
\label{def:RT}
Given System \eqref{eqn:ODE}, we call $\Rad(\theta)$ and $\Tang(\theta)$ of Theorem \ref{theo:RadTang} the radial and tangential functions of the system, respectively.  
\end{defn}
\par 
\Remark
The $r$-independence of the radial and tangential  functions is a consequence of the linearity of the system; along radial lines in $\R^2$, the vector field is simply scaled by the radius.   
\par \medskip \noindent 
\textbf{$\Rad$ and $\Tang$ capture the dynamics of System \ref{eqn:ODE} on $S^1$.}f
Let $S^1 \subset \R^2$ be the  unit circle, $r=1$. Leveraging the $r$-independence of $\Rad(\theta)$ and $\Tang(\theta)$, the second cornerstone of the paper is given by Corollary \ref{cor:RadTangS^1}, in which we observe that on $S^1$, $\Rad(\theta)$ and $\Tang(\theta)$ give precisely the radial and tangential velocities of solutions to System \eqref{eqn:ODE}. This means we can study the reactivity dynamics in all of $\R^2$ through the single-variable functions $\Rad$ and $\Tang$.
\begin{cor}
\label{cor:RadTangS^1}
Given a solution $X(t) = \begin{pmatrix} r(t)\cos(\theta(t)) \\ r(t)\sin(\theta(t)) \end{pmatrix}$ of System \eqref{eqn:ODE}, 
the radial velocity, $dr/dt$, of $X(t)$ is given by 
\[
\frac{dr}{dt} = r\Rad(\theta), 
\]
and the tangential velocity, $rd\theta/dt$, of $X(t)$ is given by 
\[
r\frac{d\theta}{dt} = r\Tang(\theta).
\]
Whenever $X(t) \in S^1$, the angular velocity, $d\theta/dt$, coincides with the tangential velocity, and
\begin{align*}
\frac{dr}{dt} & = \Rad(\theta) \\
\frac{d\theta}{dt} & =\Tang(\theta).
\end{align*}
\end{cor}
\begin{proof}
By Theorem \ref{theo:RadTang}, for a solution $X(t)=\begin{pmatrix} r(t)\cos(\theta(t)) \\ r(t)\sin(\theta(t)) \end{pmatrix}$ of System \eqref{eqn:ODE}, 
\begin{equation*}
\frac{dX}{dt} = AX =
\Rad(\theta)X+\Tang(\theta)X^\perp 
= r \Rad(\theta) \begin{pmatrix} \cos\theta \\ \sin\theta \end{pmatrix} 
+ r\Tang(\theta) \begin{pmatrix} -\sin\theta \\ \cos\theta \end{pmatrix}. 
\end{equation*}
Simultaneously, by the chain rule,
\begin{equation*}
    \frac{dX}{dt} = \frac{dr}{dt}\begin{pmatrix}
        \cos\theta\\\sin\theta
    \end{pmatrix}
    +r\frac{d\theta}{dt}\begin{pmatrix}
        -\sin\theta\\
        \cos\theta
    \end{pmatrix}.
\end{equation*}
Equating coefficients, we have 
\begin{align}
\label{eqn:dr}
\frac{dr}{dt} & = r\Rad(\theta) \\
\label{eqn:dtheta}
r\frac{d\theta}{dt} & =r\Tang(\theta) .
\end{align}
The result for $X \in S^1$ follows by setting $r=1$.
\end{proof}
\par
\textbf{$\Rad$ and $\Tang$ determine $A$.}
In Proposition \ref{prop:RadTang-to-A}, we show that we can recover the matrix $A$ from the functions $\Rad$ and $\Tang$. The power of this proposition is 
that System \eqref{eqn:ODE} can be uniquely defined either by the matrix $A$ (i.e.\ the components $a_{11}, a_{12}, a_{21}, a_{22}$) or by the pair of functions $\Rad$ and $\Tang$ (i.e.\ the parameters $m_R, m_T, p, \theta_R$).
\begin{prop}
\label{prop:RadTang-to-A}
Given real 
$A=\begin{pmatrix} a_{11} & a_{12} \\ a_{21} & a_{22} \end{pmatrix}$, 
with radial and tangential functions $\Rad(\theta)$ and $\Tang(\theta)$ as in Theorem \ref{theo:RadTang},
\begin{equation*}
A=\begin{pmatrix} \Rad(0) & -\Tang(\pi/2) \\ \Tang(0) & \Rad(\pi/2) \end{pmatrix}.
\end{equation*}
\end{prop}
\begin{proof}
Consider
$X=\begin{pmatrix} \cos\theta \\ \sin\theta \end{pmatrix} \in S^1$. 
At $\theta = 0$, $X=\begin{pmatrix} 1 \\ 0 \end{pmatrix}$ and 
$X^\perp=\begin{pmatrix} 0 \\ 1 \end{pmatrix}$.
Thus, by Theorem \ref{theo:RadTang}, 
\begin{equation*}
A \begin{pmatrix} 1 \\ 0 \end{pmatrix}
= \Rad(0) \begin{pmatrix} 1 \\ 0 \end{pmatrix} +
\Tang(0) \begin{pmatrix} 0 \\ 1 \end{pmatrix}
=\begin{pmatrix} \Rad(0) \\ \Tang(0) \end{pmatrix}.
\end{equation*}
Similarly, at $\theta = \pi/2$, 
$X=\begin{pmatrix} 0 \\ 1 \end{pmatrix}$ and 
$X^\perp=\begin{pmatrix} -1 \\ 0 \end{pmatrix}$.
Thus,
\begin{equation*}
A \begin{pmatrix} 0 \\ 1 \end{pmatrix}
= \Rad(\pi/2) \begin{pmatrix} 0 \\ 1 \end{pmatrix} +
\Tang(\pi/2) \begin{pmatrix} -1 \\ 0 \end{pmatrix}
=\begin{pmatrix} -\Tang(\pi/2) \\ \Rad(\pi/2) \end{pmatrix}.
\end{equation*}
Simultaneously, we know that
\begin{equation*}
A \begin{pmatrix} 1 \\ 0 \end{pmatrix} = 
\begin{pmatrix} a_{11} \\ a_{21} \end{pmatrix}, \  \mbox{and} \ 
A\begin{pmatrix} 0 \\ 1 \end{pmatrix} =
\begin{pmatrix} a_{12} \\ a_{22} \end{pmatrix}.
\end{equation*}
The proposition follows.
\end{proof}
\par \medskip \noindent 
\textbf{Sinusoidal properties of $\Rad$ and $\Tang$.}
Figures \ref{fig:RTdef}--\ref{fig:mT-zero-series1} and \ref{fig:reac-increase-series3}--\ref{fig:std_forms} show the graphs of $\Rad(\theta)$ and $\Tang(\theta)$ for a variety of matrices $A$. The striking sinusoidal characteristics of $\Rad$ and $\Tang$ are gathered together in Corollary \ref{cor:sinusoidal}.
\begin{cor}[of Theorem \ref{theo:RadTang}]
\label{cor:sinusoidal}
The functions $\Rad(\theta)$ and $\Tang(\theta)$ are sinusoidal with period $\pi$.  They have the same amplitude, $p$, and are $\frac{\pi}{4}$ out of phase.
The midlines of $\Rad$ and $\Tang$ are given by $m_R$ and $m_T$, respectively. The maxima of $\Rad$ and $\Tang$ occur at $\theta_R$ and $\theta_T$, respectively, where $\theta_R$ is defined in Equation \eqref{eqn:theta_R} and $\theta_T = \theta_R - \tfrac{\pi}{4}\in \R\Mod{\pi}$.
The maximum and minimum values of $\Rad$ and $\Tang$ are: 
\begin{align*}
\rho_1 & = \max_{\theta \in \R} \Rad(\theta) = \Rad(\theta_R) = m_R+p,\\
\rho_2 & = \min_{\theta \in \R} \Rad(\theta) = \Rad\left(\theta_R \pm \frac{\pi}{2}\right) = m_R-p, \\
\tau_1 & = \max_{\theta \in \R} \Tang(\theta) = \Tang\left(\theta_T\right) = m_T+p, \\
\tau_2 & = \min_{\theta \in \R} \Tang(\theta) = \Tang\left(\theta_T \pm \frac{\pi}{2}\right) = m_T-p.
\end{align*}
\end{cor}
\begin{proof}
The Corollary follows immediately from Equations \eqref{eqn:Rad} and \eqref{eqn:Tang} of Theorem \ref{theo:RadTang}
\end{proof}
\par 
\Remark
Corollary \ref{cor:sinusoidal} emphasizes the $\pi$ (rather than 2$\pi$) periodicity of the radial and tangential dynamics of System \ref{eqn:ODE} on $S^1$. This is another clear consequence of linearity, since $A(-X)=-(AX)$ for all $X \in \mathbb{R}^2$.
\par \medskip 
\textit{Resources.} See Table \ref{tab:notation} for links to interactive Desmos pages for graphing $\Rad, \Tang$.

%% file: RadTang-3-Recast.tex
\section{Recasting reactivity and attenuation in terms of \texorpdfstring{$\Rad$}{R} and \texorpdfstring{$\Tang$}{T}}
\label{sec:recast}
We now have direct access to the radial dynamics of System \eqref{eqn:ODE} through the radial function $\Rad(\theta)$. In this section, we leverage the sinusoidal structure of $\Rad$ to reveal more of the reactivity structure of the system. We begin, in Theorem \ref{theo:reactive-rho=rho_1}, by recasting reactivity and attenuation as the maximum and minimum values of $\Rad$, respectively.
\begin{theo}
\label{theo:reactive-rho=rho_1}
System \eqref{eqn:ODE} has reactivity $\rho_1=m_R+p$ and attenuation $\rho_2=m_R-p$. 
\end{theo}
\begin{proof}
By Definition \ref{def:reactivity}, the linearity of the system, and Corollaries \ref{cor:RadTangS^1} and \ref{cor:sinusoidal}, the reactivity of System \eqref{eqn:ODE} is given by
\begin{equation*}
\rho= \max_{X_0 \neq 0} \left(  \left. \frac{1}{r_0} \frac{dr}{dt} \right|_{t=0} \right)
= \max_{X_0 \in S^1} \left(\left. \frac{dr}{dt} \right|_{t=0}\right)
= \max_{\theta \in \R} \Rad(\theta)
= \rho_1 
=m_R+p.
\end{equation*}
Similarly, the attenuation of System \eqref{eqn:ODE} is given by
\begin{equation*}
\min_{X_0 \neq 0} \left(  \left. \frac{1}{r_0} \frac{dr}{dt} \right|_{t=0} \right)
= \min_{\theta \in \R} \Rad(\theta)
= \rho_2
=m_R-p.
\end{equation*}
\end{proof}
\par
\Remark
It is straightforward to see that Theorem \ref{theo:reactive-rho=rho_1} agrees with the established calculation of reactivity as the largest eigenvalue
of the Hermitian part of $A$ (Equation \eqref{eqn:reactivity-calc}) by using Equations \eqref{eqn:m_R} and \eqref{eqn:p} for $m_R$ and $p$.
\par \medskip \noindent 
\textbf{Criteria for reactivity.} 
Continuing to leverage the sinusoidal structure of $\Rad$, in Theorem \ref{theo:reactivity-criteria} we  
find simple criteria, in terms of the parameters $m_R$ and $p$ of $\Rad,$ under which System \ref{eqn:ODE} is reactive (or attenuating). 
\begin{theo}
\label{theo:reactivity-criteria}
Given System \eqref{eqn:ODE},
\begin{enumerate}[label=(\roman*),labelindent=0.9\parindent,labelsep=1pt,align=left,leftmargin=*]
\item The system is reactive if and only if $\rho_1>0$. It is attenuating if and only if $\rho_2<0$.
\item If $m_R > 0$, the system is reactive. If $m_R < 0$, the system is attenuating. 
\item If $m_R \leq 0$ ($m_R \geq 0$), then the system is reactive (attenuating) if and only if $p > |m_R|$.
\item If the system has an attracting (repelling) equilibrium at the origin, then the system is reactive (attenuating) if and only if $p > |m_R|$.
\end{enumerate}
\end{theo}
\begin{proof}
Recall from Section \ref{sec:Intro} that System \eqref{eqn:ODE} is defined to be {\it reactive} if the reactivity is positive, and {\it attenuating} if the attenuation is negative. Part $(i)$ then follows immediately from Theorem \ref{theo:reactive-rho=rho_1}.
For part $(ii)$, we know $p \geq 0$ by definition (Equation \eqref{eqn:p}). Thus, if $m_R > 0$, then $\rho_1 = m_R+p > 0$, and System \eqref{eqn:ODE} is reactive. The analogous statement in the case $m_R < 0$ is proved similarly. 
For part $(iii)$, if $m_R \leq 0$, then $\rho_1 = m_R+p > 0$ if and only if $p > |m_R|$. The analogous statement in the case $m_R \geq 0$ is proved similarly.
\par 
For part $(iv)$, note that 
\[
m_R = \frac{1}{2}\tr(A) = \frac{1}{2}(\lambda_1+\lambda_2)
\]
where $\tr(A)$ denotes the trace of $A$, and $\lambda_1, \lambda_2$ denote the eigenvalues of $A$. If System \eqref{eqn:ODE} has an attracting equilibrium at the origin, then both eigenvalues have negative real part, and $m_R < 0$. Thus, by part $(iii)$, the system is reactive if and only if $p > |m_R|$. The case when System \eqref{eqn:ODE} has a repelling equilibrium is proved similarly.
\end{proof}
\par 
\Remark
While we are most interested in the cases of reactive attractors and attenuating repellers described by parts $(iii)$ and $(iv)$, Theorem \ref{theo:reactivity-criteria} also applies to other equilibria, confirming that all systems with repellers are reactive ($m_R > 0$) and all systems with attractors are attenuating ($m_R < 0$). Systems with saddle equilibria are both reactive and attenuating, since the eigenvalues are of opposite signs, and thus $\rho_1>0$ and $\rho_2<0$.
\par \medskip \noindent 
\textbf{Reactive and attenuating regions.}
The third cornerstone of the paper is given by Definition \ref{def:ReactiveRegion}, in which we leverage the sinusoidal structure of $\Rad$ again, to decompose $\R^2$ into {\it reactive} and {\it attenuating} regions on which $r$ is monotonically increasing or decreasing, respectively, along trajectories of System \eqref{eqn:ODE}. See Figure \ref{fig:ReactiveRegion}.
This organization of the dynamics into reactive and attenuating regions is key to results throughout the paper: in Sections \ref{sec:eigen} and \ref{sec:ortho}, it underlies the duality between the eigen- and ortho-structures of the system; in Section \ref{sec:SF}, it inspires new standard forms for the matrix $A$ (Theorem \ref{theo:std_forms}); in Section \ref{sec:BoundingReactivity}, we use it to quantify the maximal amplification possible in the system (Theorems \ref{theo:MaxAmpUpperBound} and \ref{theo:MaxAmpFormula}); and in Section \ref{sec:nonaut}, it shows us how to `surf' the reactivity to infinity in nonautonomous systems (Theorem \ref{theo:non-aut_result}).
Much of the work in this paper was inspired by Josi\'{c} and Rosenbaum \cite{josic2008unstable}, who depict the reactive region in several of their figures, although they do not give it a name.
\par 
\begin{figure}
    \centering
    \includegraphics[width=\textwidth]{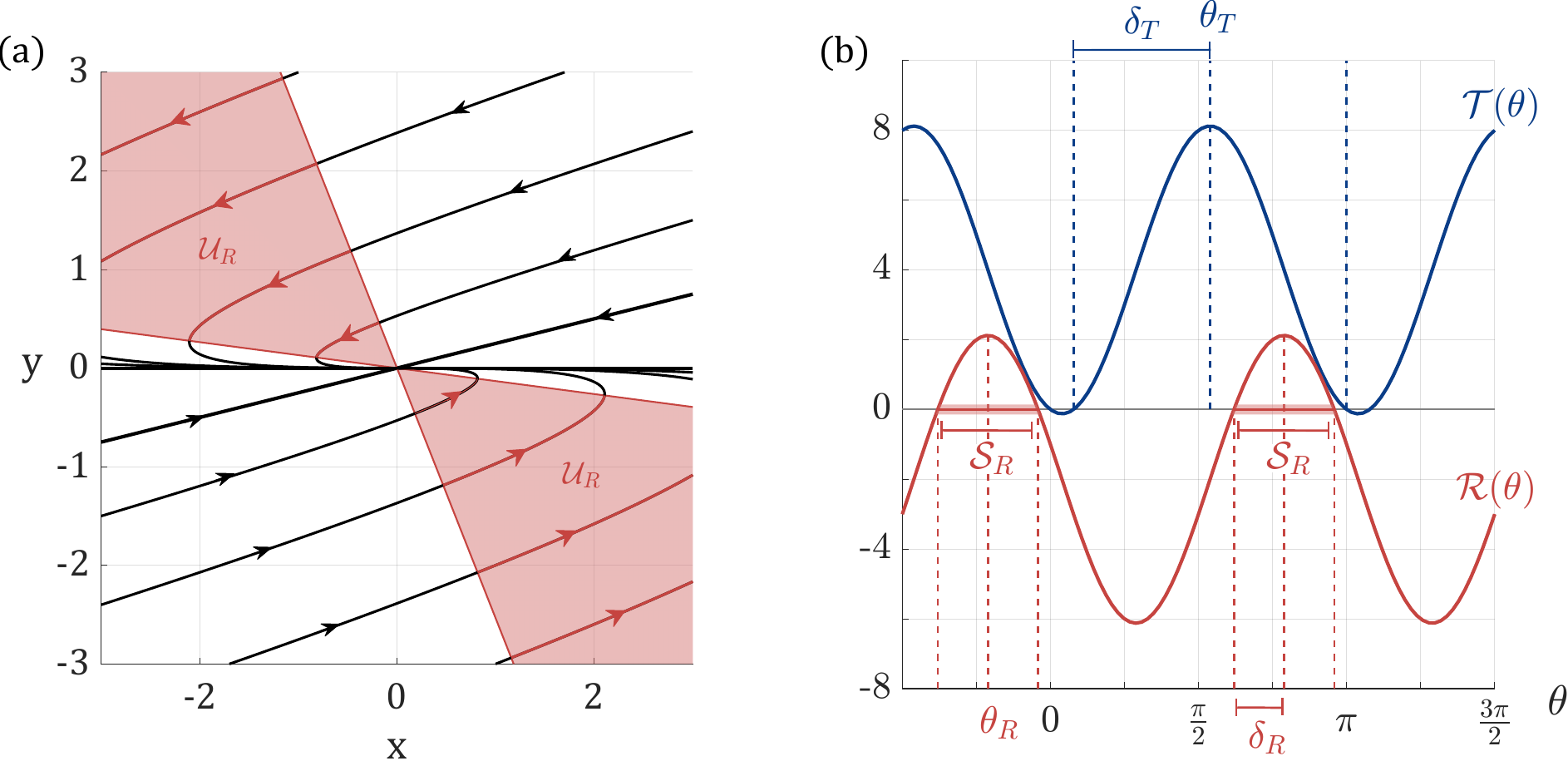}
    \caption{\textbf{(a)} The phase portrait for System \eqref{eqn:ODE} with coefficient matrix
    $A=\left(\begin{smallmatrix} -1 & -8\\ 0 & -3\end{smallmatrix}\right)$, 
    with the reactive region $\Rreg$ highlighted in red.
    \textbf{(b)} The corresponding radial and tangential functions, $\Rad$ and $\Tang$, plotted in red and dark blue, respectively. The reactive set $\Rset$, where $\Rad(\theta)>0$ is marked in red.
    The zeros of $\Rad$ are symmetrically located at distance $\delta_R$ from $\theta_R$, so that $\Rset = (\theta_R -\delta_R, \theta_R + \delta_R)\in \R\Mod{\pi}$. Similarly, the zeros of $\Tang$ are symmetrically located at distance $\delta_T$ from $\theta_T$.
    }
    \label{fig:ReactiveRegion}
\end{figure}
\par 
\begin{defn}
\label{def:ReactiveRegion}
The reactive set, $\Rset$, and attenuating set, $\Aset$, of System \eqref{eqn:ODE} are defined by:
\begin{align*}
\Rset = \left \{\theta\in [0,2\pi) \ | \ \Rad(\theta) > 0 \right \},\ \ 
\Aset = \left \{\theta\in [0,2\pi) \ | \ \Rad(\theta) < 0 \right \}.
\end{align*}
The reactive region, $\Rreg$, and attenuating region, $\Areg$, are the cones defined by the reactive and attenuating sets, respectively: 
\begin{align*}
\Rreg = \left \{(r,\theta)\in \R^2 \ | \ r >0, \ \theta \in \Rset \right \}, \ \ 
\Areg = \left \{(r,\theta)\in \R^2 \ | \ r >0, \ \theta \in \Aset \right \}.
\end{align*}
For $X \in \R^2$, we say System \eqref{eqn:ODE} is reactive at $X$ if $X \in \Rreg$, and attenuating at $X$ if $X \in \Areg$. 
\end{defn}
\par 
\Remark
The subscript $C$ is chosen for the attenuating region to denote the contracting nature of attenuation.   
\par \medskip \noindent  
\textbf{Boundary of the reactive and attenuating regions.}
If $\Rad$ is not identically zero, then $\Rset$ and $\Aset$ share the same (possibly empty) boundary, given by the zeros of $\Rad$: 
$\partial \Rset = \partial \Aset = \{ \theta \in [0,2\pi) \ | \ \Rad(\theta)=0 \}$, where $\partial S$ denotes the boundary of set $S$. Similarly, $\partial \Rreg = \partial \Areg = \{ (r,\theta) \in [0,\infty) \times [0,2\pi) \ | \ \Rad(\theta)=0 \}$. 
\par
In Definition \ref{defn:deltas}, we leverage the sinusoidal structure of $\Rad$ once again, namely the fact that zeros are symmetrically located around the maximizing angle, $\theta_R$, to define the reactivity radius, $\delta_R$. See Figure \ref{fig:ReactiveRegion}. We also define an analogous quantity $\delta_T$, for $\Tang$, whose relevance for the eigenstructure of $A$ will be made clear in Section \ref{sec:eigen}.
\par
\begin{defn}
\label{defn:deltas}
If $\Rad\not\equiv 0$ has any zeroes, we define
the reactivity radius,
$\delta_R \geq 0$, by 
\begin{align*}
&\Rad(\theta_R-\delta_R)  =\Rad(\theta_R+\delta_R)=0 \\
\mbox{ and } \ &\Rad(\theta)>0 \mbox{ on } (\theta_R-\delta_R,\theta_R+\delta_R)  \subset   \R\Mod{\pi},
\end{align*}
so that $\Rset = (\theta_R-\delta_R, \theta_R+\delta_R)\subset \R\Mod{\pi}$.
Similarly, if $\Tang\not\equiv 0$ has any zeroes, they are symmetrically located around $\theta_T$, and we define $\delta_T \geq 0$ by 
\begin{align*}
&\Tang(\theta_T-\delta_T) =\Tang(\theta_T+\delta_T)=0 \\
\mbox{ and } \ &\Tang(\theta)>0 \mbox{ on } (\theta_T-\delta_T,\theta_T+\delta_T) \subset \R\Mod{\pi}.
\end{align*}
\end{defn}
\par 
\Remark
Looking ahead, Corollary \ref{cor:delta_R_T} derives simple formulae for $\delta_R$ and $\delta_T$ using the standard matrix forms introduced in Section \ref{sec:SF}.
\par \medskip 
The reason the reactive and attenuating regions are key to so many of the results in the paper is because they clarify the way trajectories of System \eqref{eqn:ODE} accumulate (or lose) radial amplification over time. As illustrated in Figures \ref{fig:ReactiveRegion} and \ref{fig:mT-decrease-series2}, the modulus, $r(t)$, of a trajectory increases as it passes through the reactive region (marked in red), and decreases as it pass through the attenuating region. This means, for example, that if $A$ has a positive (negative) real eigenvalue, the corresponding eigenvector lies in the reactive (attenuating) region.
The angular movement of trajectories through the reactive and attenuating regions is governed by $\Tang(\theta)$.
In the next section we describe how to read more features of the eigenstructure of $A$ from $\Rad$ and $\Tang$. 

%% file: RadTang-4-Eigen.tex
\section{Eigenvalues and eigenvectors through the lens of  \texorpdfstring{$\Rad$}{R} and \texorpdfstring{$\Tang$}{T}}
\label{sec:eigen}
In Theorem \ref{theo:eigenstructure}, we show how $\Rad$ and $\Tang$ elegantly capture the eigenvectors and eigenvalues of $A$. This paves the way for the fourth cornerstone of the paper: a dual structure of orthovectors and orthovalues of $A$, defined in Section \ref{sec:ortho}.
In Corollary \ref{cor:eigenvalues}, we develop simple criteria, in terms of the parameters $m_T$ and $p$ of $\Tang$, for real or complex eigenvalues. Then, in Proposition \ref{prop:sym}, we confirm that symmetry of the matrix $A$ rules out reactive attractors and attenuating repellers.
\begin{theo}
\label{theo:eigenstructure} 
Given a real $2\times2$ matrix, $A$, with radial and tangential functions $\Rad$ and $\Tang$, and with $m_R, m_T, p, \theta_T, \delta_T$ and $\tau_i$ as summarized in Table \ref{tab:notation},
\vspace{0.5em}
\begin{enumerate}[label=(\roman*),labelindent=0.9\parindent,labelsep=1pt,align=left,leftmargin=*]
\item
The vector $V= \begin{pmatrix} \cos\theta \\ \sin\theta \end{pmatrix} \in S^1$ is an eigenvector of $A$ with corresponding eigenvalue $\lambda=\Rad(\theta)$ if and only if $\Tang(\theta)=0$.
\item
If $A$ has real eigenvalues and $\Tang\not\equiv 0$, then the eigenvectors are at 
\begin{equation*}
 \theta_1 = \theta_T+\delta_T \Mod{\pi} \ \mbox{ and } \ \theta_2 = \theta_T -\delta_T \Mod{\pi},   
\end{equation*}
with corresponding eigenvalues $\lambda_i=\Rad(\theta_i)$. 
\item
The eigenvalues of $A$, whether real or complex, are given by
\begin{equation*}
\lambda_1 = m_R + p_R \ \mbox{ and } \ \lambda_2 = m_R - p_R,
\end{equation*}
where $p_R$ is given by
\begin{equation}\label{eqn:pR-from-taus}
    p_R = \sqrt{p^2-m_T^2} = \sqrt{-\tau_1\tau_2}.
\end{equation}
\end{enumerate}
\end{theo}
\begin{proof}
For part $(i)$, by Equation \eqref{eqn:decompose}, we have
\begin{equation*}
\Tang(\theta)=0 \iff AV = \Rad(\theta)V+\Tang(\theta)V^\perp = \Rad(\theta)V,  
\end{equation*}
and the result follows. Part $(ii)$ follows immediately from part $(i)$ and the definition of $\delta_T$ (Definition \ref{defn:deltas}).
For part $(iii)$, we use the characteristic polynomial to find the eigenvalues of $A$. Recall, from Proposition \ref{prop:RadTang-to-A}, that $A$ can be written in terms of $\Rad$ and $\Tang$ as follows:
\begin{equation*}
A=\begin{pmatrix} \Rad(0) & -\Tang(\pi/2) \\ \Tang(0) & \Rad(\pi/2) \end{pmatrix}
\end{equation*}
Using equations \eqref{eqn:Rad} and \eqref{eqn:Tang}, a straightforward calculation shows that the trace and determinant of $A$ are given by 
    \begin{align*}
        \text{tr}(A) = 2 \,m_R \qquad \mbox{ and } \qquad \text{det}(A) = m_R^2 + m_T^2 - p^2.
    \end{align*}
    Thus, the characteristic polynomial of $A$ is given by
    \begin{align*}
        \lambda^2 - 2m_R\lambda + m_R^2 + m_T^2 - p^2 = 0,
    \end{align*}
    yielding roots
    \begin{align}
        \lambda_{1,2} &= m_R \pm \sqrt{p^2-m_T^2} \label{eqn:pR-from-p-mT}\\
        & = m_R \pm \sqrt{-(m_T+p)(m_T-p)} = m_R \pm \sqrt{-\tau_1\tau_2} = m_R \pm p_R.\notag
    \end{align}
Finally, the direction of the $\pi/4$ phase shift between $\Rad$ and $\Tang$ ensures that, when the eigenvalues are real, the subscripts are consistent so that $\lambda_i = \Rad(\theta_i)$ as in Figure \ref{fig:saddle}. 
\end{proof}
\begin{figure}[t!]
	\centering
    \includegraphics[width=\textwidth]{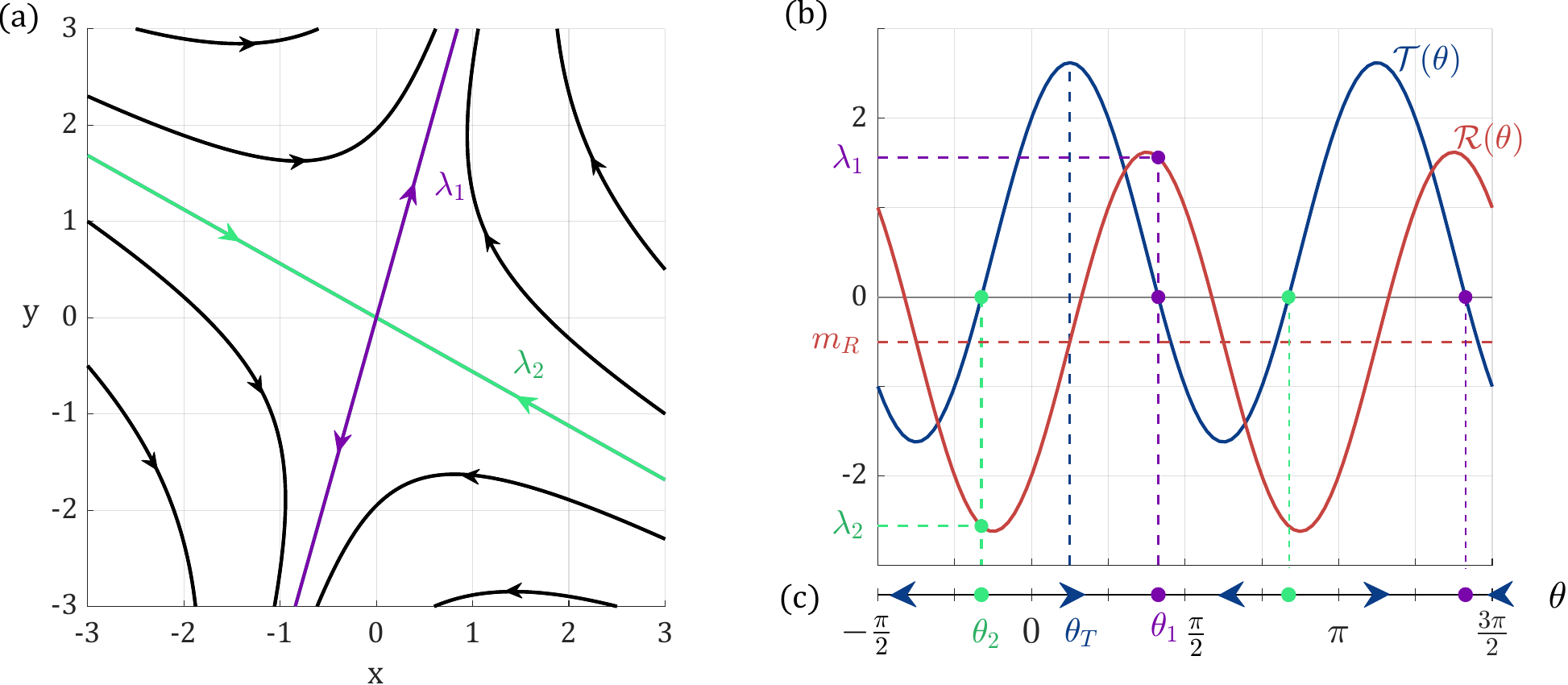}
    \setlength{\belowcaptionskip}{-5pt}
	\caption{\textbf{(a)} The phase portrait for 
    System \eqref{eqn:ODE}
    with coefficient matrix $A=\left(\begin{smallmatrix}-2&1\\2&1 \end{smallmatrix}\right)$ --- the same example as in Figure \ref{fig:RTdef}. 
    The origin is a saddle 
    equilibrium
    whose eigendirections
    are highlighted in dark purple (repelling direction) 
    and light green (attracting direction). 
    The associated eigenvalues are $\lambda_1>0$ and $\lambda_2<0$.
    \textbf{(b)} The radial and tangential 
    functions, $\Rad$ and $\Tang$,
    plotted in red and dark blue, respectively. 
    The eigendirections in (a) are represented by
    angles $\theta_1$ and $\theta_2$ in (b) and the 
    corresponding eigenvalues $\lambda_1$ and $\lambda_2$ 
    are given by $\Rad(\theta_1)$ and $\Rad(\theta_2)$, respectively. 
    \textbf{(c)} 
    The phase line of 
    the angular dynamics (Equation \eqref{eqn:dtheta}) which highlights how trajectories approach the eigendirection, $\theta_1$, of the larger eigenvalue
    in their long-term behavior.}
	\label{fig:saddle}
\end{figure}
\par \medskip \noindent 
\textbf{The case of distinct real eigenvalues.}
Recall that $\tau_1, \tau_2$ are the maximum and minimum values of $\Tang$ respectively. Thus, if $\tau_1 > 0 > \tau_2$, then $\Tang$ has distinct zeros, $p_R>0$, and the eigenvalues are real and distinct. 
Note also that $m_R=\Rad(\theta_T)$ by the $\pi/4$ phase shift between $\Rad$ and $\Tang$, and so the symmetric spacing of the eigenvalues, $\lambda_i$, relative to the midline $m_R$ of $\Rad$ reflects the symmetric spacing of the eigenvectors, $\theta_i$, relative to the maximizing angle, $\theta_T$ of $\Tang$. See Figure \ref{fig:saddle}b.
\par 
Figure \ref{fig:saddle}c illustrates that in the decoupled, one-dimensional angular dynamics of Equation \eqref{eqn:dtheta}, the eigendirection, $\theta_1$, associated with the larger eigenvalue, $\lambda_1$, is attracting (since  $\Tang'(\theta_1) < 0 < \Tang'(\theta_2)$, by the definition of $\delta_T$). This corresponds to the familiar asymptotic behavior in the phase plane, where trajectories eventually follow the weak stable direction when the origin attracts, the unstable direction when the origin is a saddle, and the strong unstable direction when the origin repels.
\par \medskip \noindent 
\textbf{The case of complex conjugate eigenvalues.}
If $\tau_1, \tau_2$ are the same sign, then $\Tang$ has no zeros, $p_R$ is imaginary, and there is a complex conjugate pair of eigenvalues In this case, the real part of the eigenvalues is given by the arithmetic mean, $m_R$, of the maximum and minimum radial velocities, while the imaginary part is given by the geometric mean, $\sqrt{\tau_1\tau_2}$, of the maximum and minimum angular velocities. Thus, in the phase plane, trajectories of System \eqref{eqn:ODE} move counter-clockwise ($\Tang>0$) or clockwise ($\Tang<0$), with angular period $2\pi/\sqrt{\tau_1\tau_2}$ (Figures \ref{fig:mT-decrease-series2}a and  \ref{fig:mR-decrease-series4}).
\par \medskip \noindent  
\textbf{Criteria for real versus complex conjugate eigenvalues.}
The following corollary of Theorem \ref{theo:eigenstructure} classifies the possible eigenvalues of $A$. The result leverages the fact that the relative sizes of the midline $m_T$ and amplitude $p$ of the sinusoidal function $\Tang$ easily determine whether or not $\Tang$ is single-signed.
\begin{figure}[b!]
    \centering
    \includegraphics[width=\textwidth]{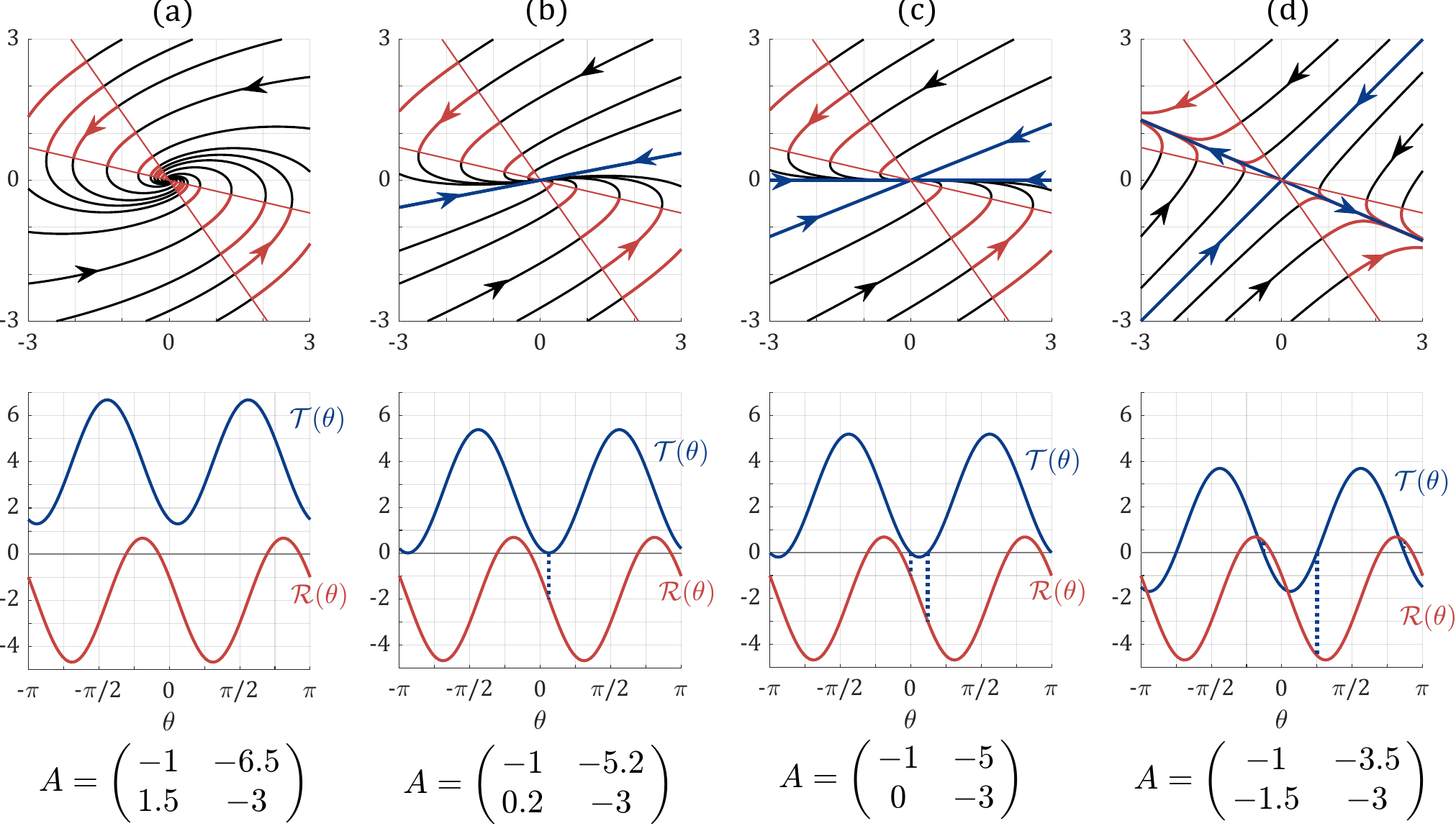}
    \caption{Phase portraits of System \eqref{eqn:ODE}, and corresponding graphs of $\Rad$ and $\Tang$, associated with the matrices shown. The series illustrates Theorem \ref{theo:eigenstructure} and Corollary \ref{cor:eigenvalues}: eigenvectors are given by the roots, $\theta_i$, of the tangential function $\Tang$, and have corresponding eigenvalues $\Rad(\theta_i)$. All four systems have the same radial function $\Rad$ and are reactive. The reactive region is shown in red on the phase plane.
    The difference between the systems is that  $\Tang$ is shifted down through the series, yielding different numbers of roots, and hence qualitatively different dynamics. 
    \textbf{(a)} $m_T>p$, $\Tang$ has no zeros, and so there are no real eigenvalues. Dynamically, trajectories rotate counter-clockwise (since $\Tang >0$) and spiral inward (since $m_R<0$). Along the way, they exhibit transient radial amplification as they pass through the reactive region, and so the system has a reactive spiral sink.
    \textbf{(b)} $m_T=p$, $\Tang$ has exactly one root (mod $\pi$), and there is a single negative eigenvalue.
    Dynamically, every trajectory not on the eigenline exhibits counter-clockwise rotation towards the eigenline as it approaches the origin. 
    \textbf{(c)} $m_T < p$, $\Tang$ has two roots (mod $\pi$), and there are distinct eigenvalues.  
    Since $\Tang$ just barely passes below the $\theta$-axis, the eigenvectors are close together, both eigenvalues are negative, and the system has a reactive attracting node. 
    \textbf{(d)} When $m_T$ is shifted down further, the roots of $\Tang$ get farther apart and one eigenvector passes into the reactive region. The corresponding eigenvalue is positive and the system has a saddle equilibrium.
    }
    \label{fig:mT-decrease-series2}
\end{figure}
\begin{cor}
\label{cor:eigenvalues} 
Given a real $2\times2$ matrix, $A$, with $\Rad, \Tang, m_R, m_T, p$ and $\tau_i$ as summarized in Table \ref{tab:notation},
\begin{enumerate}[label=(\roman*),labelindent=0.9\parindent,labelsep=1pt,align=left,leftmargin=*]
\item
$A$ has the complex conjugate pair of distinct eigenvalues, $\lambda_{1,2} = m_R \pm i\sqrt{\tau_1\tau_2}$, if and only if $p < |m_T|$.
\item
$A$ has the distinct real eigenvalues, $\lambda_{1,2} = \Rad(\theta_{1,2}) = m_R \pm \sqrt{-\tau_1\tau_2}$, if and only if $p > |m_T|$.
\item
$A$ has the single real eigenvalue, $\lambda=m_R$, with algebraic multiplicity 2 and geometric multiplicity  1 if and only if $p = |m_T| \neq 0$.
\item 
$A$ has the single real eigenvalue, $\lambda = m_R$, with algebraic and geometric multiplicity 2 if and only if $p = m_T = 0$. 
\end{enumerate}
\end{cor}
\begin{proof}
By the definition of $\Tang$ (Equation \eqref{eqn:Tang}), $\Tang$ is single-signed if and only if $p < |m_T|$, and has two zeros in $[0,\pi)$ if and only if $p > |m_T|$. In the case when $p = |m_T|$, $\Tang$ has a single zero in $[0,\pi)$ if and only if $p \neq 0$, and $\Tang \equiv 0$ if and only if $p = m_T = 0$. Thus, the corollary follows directly from Theorem \ref{theo:eigenstructure}. 
\end{proof}
The cases of Corollary \ref{cor:eigenvalues} are illustrated in Figures \ref{fig:mT-decrease-series2} and \ref{fig:mT-zero-series1}, where \ref{fig:mT-decrease-series2}a corresponds to case $(i)$, \ref{fig:mT-decrease-series2}c--d to case $(ii)$, \ref{fig:mT-decrease-series2}b to case $(iii)$, and \ref{fig:mT-zero-series1}a to case $(iv)$.
\begin{figure}[t!]
	\centering
	\includegraphics[width=\textwidth]{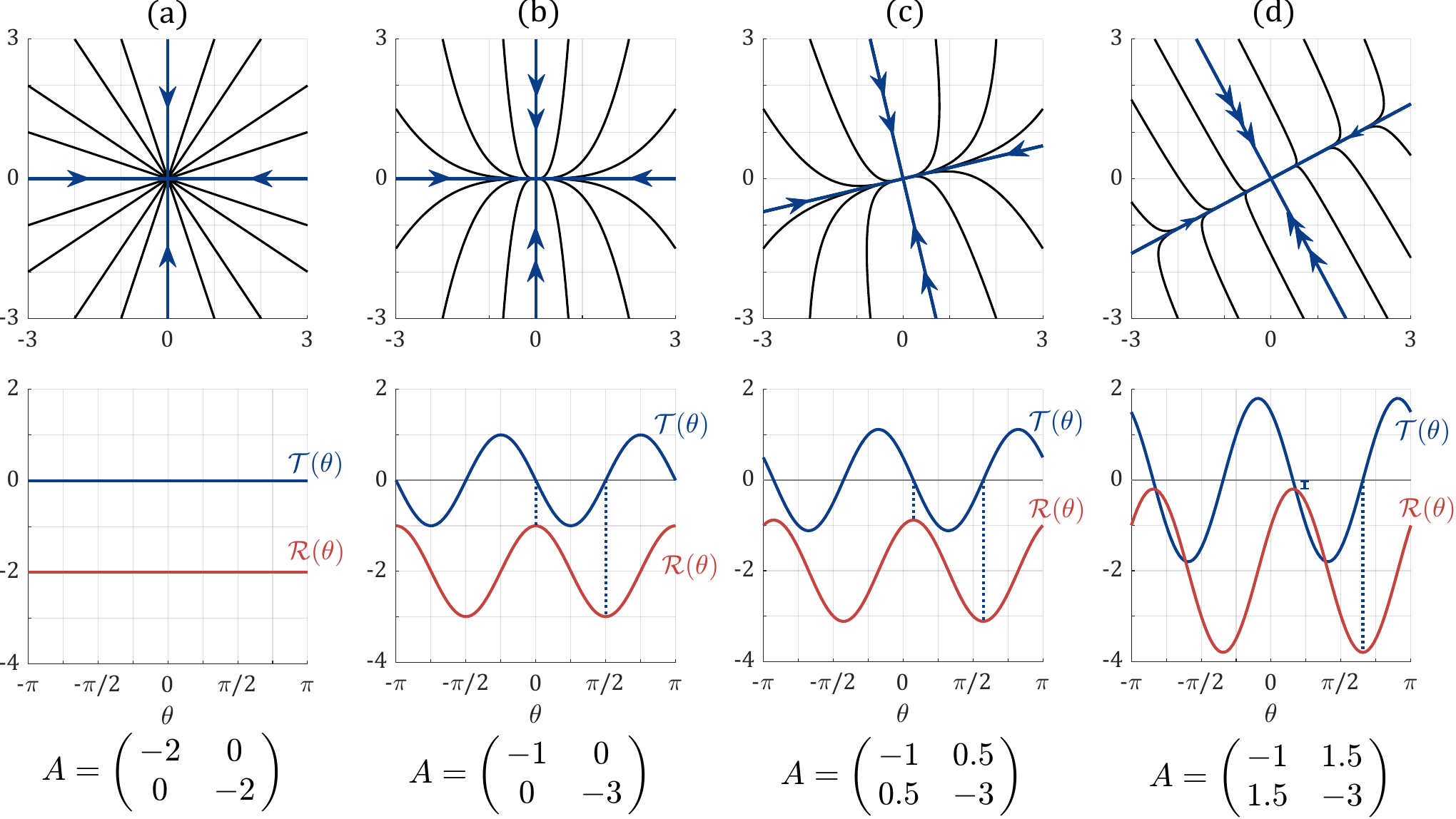}
	\caption{Phase portraits of System \eqref{eqn:ODE}, and corresponding graphs of $\Rad$ and $\Tang$, associated with the matrices shown. All four matrices are symmetric, illustrating Proposition \ref{prop:sym} and Lemma \ref{lem:sym}. 
		\textbf{(a)} $\Tang \equiv 0, \Rad \equiv -2$, and every vector is an eigenvector with eigenvalue $-2$. 
        In (b)--(d), $\Tang \not\equiv 0$ and Lemma \ref{lem:sym} applies.
		\textbf{(b)} $A$ is diagonal and has eigenvectors at the axes: $\Tang(0)=\Tang(\pi/2)=0$. 
		\textbf{(c)} The diagonal entries of $A$ are as in (b), but the off-diagonal entries are non-zero. This slightly increases the amplitude, $p$, and shifts the maximum of $\Rad$ to the right. The eigenvectors are still orthogonal but have been rotated off the coordinate axes.
		\textbf{(d)} The off-diagonal entries of $A$ are increased from (c), further increasing $p$, and further rotating the orthogonal eigenvectors. The increase in $p$ leads to an increase in eigenvalue separation seen on $\Rad$ and on the phase portrait.}
	\label{fig:mT-zero-series1}
\end{figure}
\par \medskip \noindent  
\textbf{Symmetric attractors are non-reactive.}
In Proposition \ref{prop:sym}, we consider the consequences of symmetry on reactivity and attenuation. See Figure \ref{fig:mT-zero-series1}.
\begin{prop}
\label{prop:sym} 
Given System \eqref{eqn:ODE}, if $A$ is symmetric, then the origin is neither a reactive attractor nor an attenuating repeller.
\end{prop}
Proposition \ref{prop:sym} can, of course, be proven
from the solutions of System \eqref{eqn:ODE}. Here, we show how it follows from the properties of $\Rad$ and $\Tang$ by first proving Lemma \ref{lem:sym}, in which we recast the well-known result that symmetric matrices have real eigenvalues and orthogonal eigenvectors in terms of $\Rad$ and $\Tang$.
\begin{lem}
\label{lem:sym}
Given a real $2\times2$ matrix, $A$, with $\Tang, m_R, m_T, p, \theta_R$ and $\rho_i$ as summarized in Table \ref{tab:notation},
if $\Tang \not\equiv 0$, the following statements are equivalent:
\begin{enumerate}[label=(\roman*),labelindent=0.9\parindent,labelsep=1pt,align=left,leftmargin=*]
\item $A$ is symmetric.
\item $m_T=0$.
\item $A$ has real orthogonal eigenvectors.
\item $A$ has real eigenvalues $\lambda_1=\rho_1=m_R+p$ and $\lambda_2=\rho_2=m_R-p$.
\item $A$ has 
eigenvectors $V_1$ and $V_2$
at angles $\theta_R$ and $\theta_R \pm \frac{\pi}{2}$, respectively.
\end{enumerate}
\end{lem}
\begin{proof}
$(i) \Leftrightarrow (ii)$ follows immediately from the definition of $m_T$ (Equation \eqref{eqn:m_T}).
By Equation \eqref{eqn:Tang}, $m_T$ is the midline of $\Tang$. Since $\Tang$ has period $\pi$, the zeroes of $\Tang$ lie on the midline of $\Tang$ (i.e. $m_T=0$) if and only if they are angle $\frac{\pi}{2}$ apart. Thus, $(ii) \Leftrightarrow (iii)$ follows from Theorem \ref{theo:eigenstructure}.
Moreover, the $\frac{\pi}{4}$ phase shift between $\Tang$ and $\Rad$ ensures that $\Rad$ reaches its maximum and minimum values, $\rho_1$ and $\rho_2$, respectively, exactly when $\Tang$ is at its midline. Thus, $(ii) \Leftrightarrow (iv)$ and
$(iv) \Leftrightarrow (v)$ follow from Corollary \ref{cor:sinusoidal} and Theorem \ref{theo:eigenstructure}.
\end{proof}
\begin{proof}[Proof of Proposition \ref{prop:sym}]
By Lemma \ref{lem:sym}, if $A$ is symmetric, then 
the reactivity $\rho_1$ and the attenuation $\rho_2$ are eigenvalues of $A$. 
If System \eqref{eqn:ODE} is reactive, then, by Theorem \ref{theo:reactivity-criteria}, it has reactivity $\rho_1 > 0$, and so the origin is not an attractor. Similarly, if System \eqref{eqn:ODE} is attenuating, then it has attenuation $\rho_2 < 0$, and the origin is not a repeller. 
\end{proof}
Taken together, Proposition \ref{prop:sym} and Lemma \ref{lem:sym} show that if $A$ has orthogonal real eigenvectors, then System \eqref{eqn:ODE} has neither a reactive attractor nor an attenuating repeller, reinforcing the observation that the process of diagonalization masks transient behavior \cite{higham1993stiffness}. By contrast, in Section \ref{sec:BoundingReactivity} we show that any linearly independent, non-orthogonal pair of vectors can be the eigenvectors of a reactive attractor
with any reactivity (and of an attenuating repeller
with any attenuation). 
\par \medskip \noindent  
\textbf{Radial and tangential functions of normal matrices.}
A matrix is said to be \textit{normal} if it commutes with its conjugate transpose \cite{grone1987normal}. For real $2\times2$ matrices, normal matrices are either symmetric or have the form $A = aI + bJ$, where $a, b \in \mathbb{R}$, $I$ is the identity matrix, and $J$ is the rotation matrix defined in Notation \ref{notn:J}. Proposition \ref{prop:sym} and Lemma \ref{lem:sym} show that if $A$ is symmetric, then $m_T=0$ and System \eqref{eqn:ODE} does not have a reactive attractor. 
In the case when $A$ has the form $A = aI + bJ$, the radial and tangential functions of System \eqref{eqn:ODE} are both constant: $\Rad(\theta) = a$ and $\Tang(\theta) = b$.  Thus, $\Rad$ and $\Tang$ have amplitude $p = 0$, and
part $(iv)$ of Theorem \ref{theo:reactivity-criteria} confirms that System \eqref{eqn:ODE} does not admit reactive attractors (nor attenuating repellers) when $A$ is normal.
See \cite{trefethen2020spectra} for a rich analysis of normal and non-normal systems.

%% file: RadTang-5-Ortho.tex
\section{Reactivity through the lens of orthovalues and orthovectors}
\label{sec:ortho}
In this section, we introduce the fourth cornerstone of the paper: a structure of \textit{orthovalues} and \textit{orthovectors}, dual to the eigenstructure.
To motivate the orthostructure, we recall the connection between the eigenstructure of a matrix $A$ and the radial and tangential decomposition of the vector field $AV$.
Writing $V=r\left(\begin{smallmatrix}
    \cos\phi\\\sin\phi
\end{smallmatrix}\right)$, we have, from Theorem \ref{theo:RadTang},
\[
AV = \Rad(\phi)V + \Tang(\phi)V^\perp.
\]
By Theorem \ref{theo:eigenstructure}, $V$ is an eigenvector of $A$ with eigenvalue $\Rad(\phi)$ if and only if $\Tang(\phi) = 0$. In this section, we investigate what happens when $\Rad(\phi) = 0$, which occurs when $AV$ is  orthogonal to $V$. 
\begin{defn}
\label{def:ortho}
Given a real $2\times2$ matrix, $A$,
we call $V \in \R^2$ an orthovector of $A$ with corresponding orthovalue $\mu$ if and only if
$AV = \mu V^\perp$
and $V \neq 0$. 
\end{defn}
\par 
\Remark
The term orthostructure refers to the pairs of orthovectors and orthovalues of $A$. 
\par \medskip 
In Theorem \ref{theo:ortho-dualcalc}, we highlight one aspect of the duality between eigenstructure and orthostructure, using Definition \ref{def:ortho} and the 
rotation matrix $J$ (Notation \ref{notn:J}). In doing so, we show that computing orthovalues and orthovectors can be reduced to the computation of eigenvalues and eigenvectors. In Theorem \ref{theo:orthostructure}, in direct analogy with Theorem \ref{theo:eigenstructure}, we highlight a different aspect of the duality between eigenstructure and orthostructure, deriving an approach for computing orthovalues and orthovectors directly from the radial and tangential functions $\Rad$ and $\Tang$. 
In both of Theorems \ref{theo:ortho-dualcalc} and \ref{theo:orthostructure}, we make use of the fact that multiplying $A$ by either $J$ or $-J$ results in an exchange of radial and tangential dynamics in System \eqref{eqn:ODE}; see Figure \ref{fig:duality}.
\begin{theo}
\label{theo:ortho-dualcalc} 
Given a real $2\times2$ matrix, $A$,
$V \in \R^2$ is an orthovector of $A$ with corresponding orthovalue $\mu$ if and only if $V$ is an eigenvector of $-JA$ with corresponding eigenvalue $\mu$.
\end{theo}
\begin{proof}
$V \in \R^2$ is an eigenvector of $-JA$ with corresponding eigenvalue $\mu$ if and only if 
\begin{align*}
-JAV = \mu V
&\iff -J^2AV = \mu JV\\
&\iff AV = \mu V^\perp,
\end{align*}
as $-J^2 = I$ and $JV = V^\perp $.
\end{proof}
\begin{figure}
    \centering
    \includegraphics[width=\textwidth]{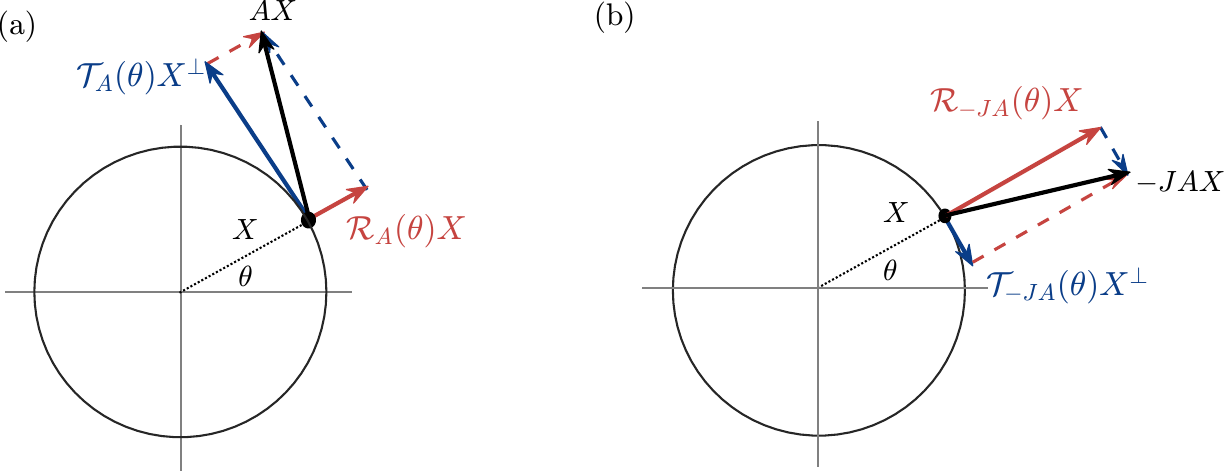}
    \caption{The radial and tangential decompositions of  \textbf{(a)} the vector field $AX$, and \textbf{(b)} the vector field $-JAX$. The tangential component, $\Tang_A(\theta)X^\perp$, of $AX$ is rotated clockwise by $-J$ to $\Rad_{-JA}(\theta)X$, while  $\Rad_A(\theta)X$ is rotated to $\Tang_{-JA}(\theta)X^\perp$. Note that, by the definition of $X^\perp$ in Notation \ref{notn:J}, $\Tang_{-JA}(\theta) < 0$ in this example, so that  
    $\Tang_{-JA}(\theta) = -\Rad_A(\theta)$ and 
    $\Rad_{-JA}(\theta) = \Tang_A(\theta)$,
    illustrating Theorem \ref{theo:ortho-dualcalc} and Equation \eqref{eqn:duality}.
    }
    \label{fig:duality}
\end{figure}
We now begin to uncover the relationship between orthostructure and reactivity. Comparing Definition \ref{def:ortho} with Definition \ref{def:ReactiveRegion}, we note that the boundary of the reactive set $S_R$ is determined by the zeros of the radial function $\Rad$.
In Theorem \ref{theo:orthostructure}, we prove that the orthovectors of $A$ determine the boundaries of the reactive region $\mathcal{U}_R$; we also establish formulae for computing orthovalues and orthovectors directly  from our radial and tangential functions.
In Section \ref{sec:SF}, we use these results to guide our construction of standard matrix forms for assessing reactivity. To streamline the exposition, we introduce the following notation for radial and tangential functions associated with a given matrix.
\begin{notn}
    The expressions $\Rad_M(\theta)$ and $\Tang_M(\theta)$ denote the radial and tangential functions  corresponding to System \eqref{eqn:ODE} with coefficient matrix $M$. Similar subscripts are occasionally appended to parameters of $\Rad$ and $\Tang$ (e.g.\ $(m_R)_M$ and $(\tau_i)_M$) to avoid ambiguity. When the coefficient matrix is clear from the context, the subscript is omitted. 
\end{notn}
\par 
\begin{theo}
\label{theo:orthostructure}
Given a real $2 \times 2$ matrix, $A$, with $\Rad, \Tang, m_R, m_T, p, \theta_R, \delta_R$ and $\rho_i$ as summarized in Table \ref{tab:notation},
\begin{enumerate}[label=(\roman*),labelindent=0.9\parindent,labelsep=1pt,align=left,leftmargin=*]
\item
The vector $V= \begin{pmatrix} \cos\phi \\ \sin\phi \end{pmatrix} \in S^1$ is an orthovector of $A$ with corresponding orthovalue $\mu=\Tang(\phi)$ if and only if $\Rad(\phi)=0$.
\item
If $A$ has real orthovalues and $\Rad\not\equiv 0$, the orthovectors are at 
\begin{equation*}
    \phi_1 = \theta_R-\delta_R \Mod{\pi} \ \mbox{ and } \ \phi_2 = \theta_R +\delta_R \Mod{\pi},
\end{equation*} 
with corresponding orthovalues $\mu_i=\Tang(\phi_i)$. 
\item
The orthovalues of $A$, whether real or complex, are given by 
\begin{equation*}
    \mu_1 = m_T + p_T \ \mbox{ and } \ \mu_2 = m_T - p_T, 
\end{equation*}
where $p_T$ is given by
\begin{equation}\label{eqn:pT-from-taus}
    p_T = \sqrt{p^2-m_R^2} = \sqrt{-\rho_1\rho_2}
\end{equation}
\end{enumerate}
\end{theo}
\begin{proof}
The proofs of parts $(i)$ and $(ii)$ are analogous to those of Theorem \ref{theo:eigenstructure} parts $(i)$ and $(ii)$.
For part $(i)$, by Equation \eqref{eqn:decompose}, we have
\begin{equation*}
\Rad(\phi)=0 \iff AV = \Rad(\phi)V+\Tang(\phi)V^\perp = \Tang(\phi)V^\perp
\end{equation*}
and the result follows from Definition \ref{def:ortho}. 
Part $(ii)$ follows immediately from part $(i)$ and the definition of $\delta_R$ (Definition \ref{defn:deltas}).
\par 
We know, from Theorem \ref{theo:ortho-dualcalc}, that the orthovalues of $A$ are the eigenvalues of $-JA$. Thus, part $(iii)$ of Theorem \ref{theo:orthostructure} is inherited from part $(iii)$ of Theorem \ref{theo:eigenstructure} as follows: 
For $X \in \R^2$, we apply Theorem \ref{theo:RadTang} to decompose the vector field $(-JA)X$ into radial and tangential components, yielding
\begin{equation}\label{eqn:ortho1}
(-JA)X = \Rad_{-JA}(\theta)X+\Tang_{-JA}(\theta)X^\perp.
\end{equation}
Simultaneously, decomposing only $AX$, we also have 
\begin{align}
-J(AX)  & =  -J \left(\Rad_A(\theta)X+\Tang_A(\theta)X^\perp\right) \nonumber\\
 & =  \Rad_A(\theta)(-JX)+\Tang_A(\theta)(-J(JX))  \nonumber\\ 
 & =  -\Rad_A(\theta)X^\perp + \Tang_A(\theta)X, \label{eqn:ortho2}
\end{align}
where we have used the facts that $X^\perp = JX$ and $-J^2 = I$. Equating the coefficients of $X$ and $X^\perp$ in Equations \eqref{eqn:ortho1} and \eqref{eqn:ortho2}, we deduce that
\begin{equation}
\label{eqn:duality}
\Rad_{-JA}(\theta) = \Tang_A(\theta) \  \mbox{ and } \ 
\Tang_{-JA}(\theta) = -\Rad_A(\theta).
\end{equation}
\par
By Theorem \ref{theo:eigenstructure}, the eigenvalues of $-JA$, and hence the orthovalues of $A$, are given by: 
\[
(m_R)_{(-JA)} \pm \sqrt{-(\tau_1)_{(-JA)} (\tau_2)_{(-JA)}}
\]
But, by equations \eqref{eqn:duality}, 
\[
(m_R)_{(-JA)}=(m_T)_A, \ \ (\tau_1)_{(-JA)}= -(\rho_2)_A, \ \ \mbox{and} \ \ (\tau_2)_{(-JA)} = -(\rho_1)_A.
\]
The result follows.
\end{proof}
\Remark
The choice to ensure $\lambda_1 \geq \lambda_2$ and $\mu_1 \geq \mu_2$, while $\lambda_i=\Rad(\theta_i)$ and $\mu_i=\Tang(\phi_i)$, leads to a slight variation in the definitions of $\phi_1,\phi_2$ in Theorem \ref{theo:orthostructure} compared with those of $\theta_1, \theta_2$ in Theorem \ref{theo:eigenstructure}. This is a consequence of the phase shift between $\Rad$ and $\Tang$.
\begin{figure}[t!]
    \centering
    \includegraphics[width=\textwidth]{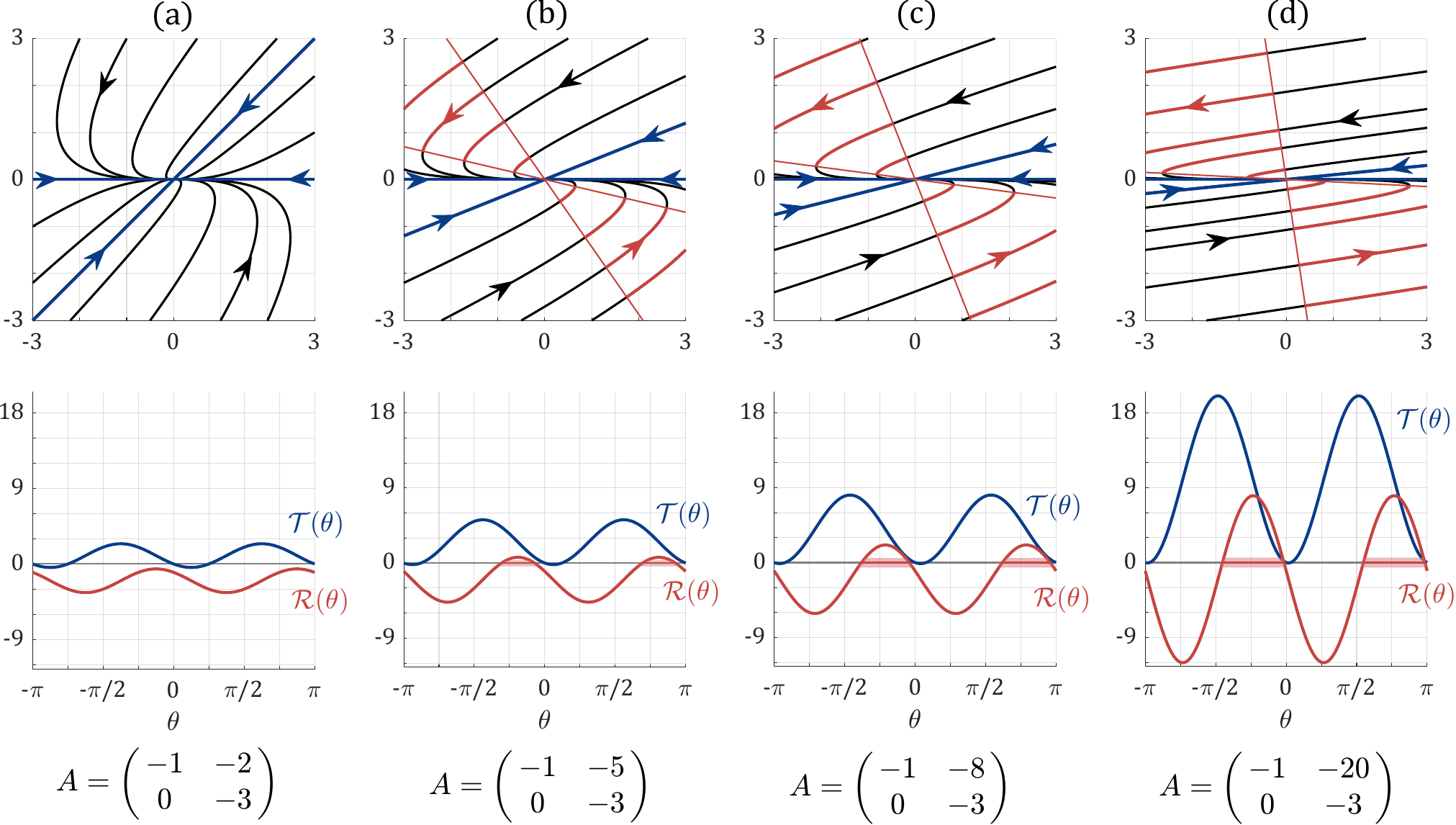}
    \setlength{\belowcaptionskip}{-4pt}
    \caption{
    Phase portraits of System \eqref{eqn:ODE}, and corresponding graphs of $\Rad$ and $\Tang$, associated with the matrices shown. 
    In all four examples, the eigenvalues and the midline $m_R$ are held constant. As the off-diagonal entry in $A$ increases in magnitude through the series, the amplitude $p$ increases, leading to increasingly reactive dynamics in which the eigenvectors converge ($\delta_T$ decreases), a reactive region emerges and then expands ($\delta_R$ increases), and the reactivity $\rho_1$ increases.
    For cross-referencing: the system in (b) is the same as in Figures \ref{fig:mT-decrease-series2}c and \ref{fig:std_forms}, and the system in (c) is the same as in Figures \ref{fig:bullseye}, \ref{fig:ReactiveRegion} and \ref{fig:MaxAmpUpperBound}.}
    \label{fig:reac-increase-series3}
\end{figure}
\par \medskip \noindent 
\textbf{Criteria for reactivity.}
The eigenstructure of $A$ provides insight into the asymptotic behavior of solutions to System \eqref{eqn:ODE}. The orthostructure, on the other hand, provides insight into the transient behavior of solutions. Since the functions $\Rad$ and $\Tang$ capture both the eigenstructure and orthostructure, they offer a unified approach to conducting asymptotic and transient analysis. 
In Corollary \ref{cor:eigenvalues}, we classified the eigenvalues, and hence the asymptotic stability of System \eqref{eqn:ODE}, using just the parameters $m_T$ and $p$ of $\Tang$. In direct analogy with this, in Corollary \ref{cor:ortho-char} we classify the orthovalues, and hence the transient behavior of System \eqref{eqn:ODE}, using the parameters $m_R$ and $p$ of $\Rad$.
\begin{cor}
\label{cor:ortho-char}
Given a real $2 \times 2$ matrix, A, with $\Rad, \Tang, m_R, m_T, p$ and $\rho_i$ as summarized in Table \ref{tab:notation},
\begin{enumerate}[label=(\roman*),labelindent=0.9\parindent,labelsep=1pt,align=left,leftmargin=*]
    \item 
    $A$ has the complex conjugate pair of distinct orthovalues, $\mu_{1,2} = m_T \pm i\sqrt{\rho_1\rho_2}$, if and only if $p < |m_R|$.
    \item 
    $A$ has the distinct real orthovalues, $\mu_{1,2} = \Tang(\phi_{1,2}) = m_T \pm \sqrt{-\rho_1\rho_2}$, if and only if $p > |m_R|$.
    \item 
    $A$ has the single real orthovalue, $\mu=m_T$, with algebraic multiplicity 2 and geometric multiplicity 1 if and only if $p = |m_R| \neq 0$.
    \item 
    $A$ has the single real orthovalue, $\mu = m_T$, with algebraic and geometric multiplicity 2 if and only if $p = m_R = 0$. 
\end{enumerate}
\end{cor}
\begin{proof}
    The proof is analogous to that of Corollary \ref{cor:eigenvalues}, and follows from Equation \eqref{eqn:Rad} and Theorem \ref{theo:orthostructure}.
\end{proof}
\par
Dynamically, case $(i)$ in Corollary \ref{cor:ortho-char} corresponds to the case when $\Rad$ has no zeroes and System \eqref{eqn:ODE} is either purely reactive or purely attenuating. Thus, the origin is either a non-attenuating repeller or a non-reactive attractor 
(as in Figures \ref{fig:mT-zero-series1}, \ref{fig:reac-increase-series3}a and \ref{fig:mR-decrease-series4}d). In case $(ii)$, System \eqref{eqn:ODE} has both reactive and attenuating regions, so the origin is either a reactive attractor 
(as in Figures \ref{fig:reac-increase-series3}b-d and \ref{fig:mR-decrease-series4}c, among others),
an attenuating repeller, a non-circular center
(as in Figure \ref{fig:mR-decrease-series4}b),
or a saddle
(as in Figures \ref{fig:saddle} and \ref{fig:mT-decrease-series2}d).
Case $(iii)$ is the edge case between cases $(i)$ and $(ii)$, where System \eqref{eqn:ODE} still has a non-reactive attractor or a non-attenuating repeller, but also a single direction on which $\Rad=0$. Case $(iv)$ describes systems with circular centers and uniform angular velocity $m_T$ (as in Figure \ref{fig:mR-decrease-series4}a). 
\par \medskip \noindent  
\begin{figure}
    \centering
    \includegraphics[width=\textwidth]{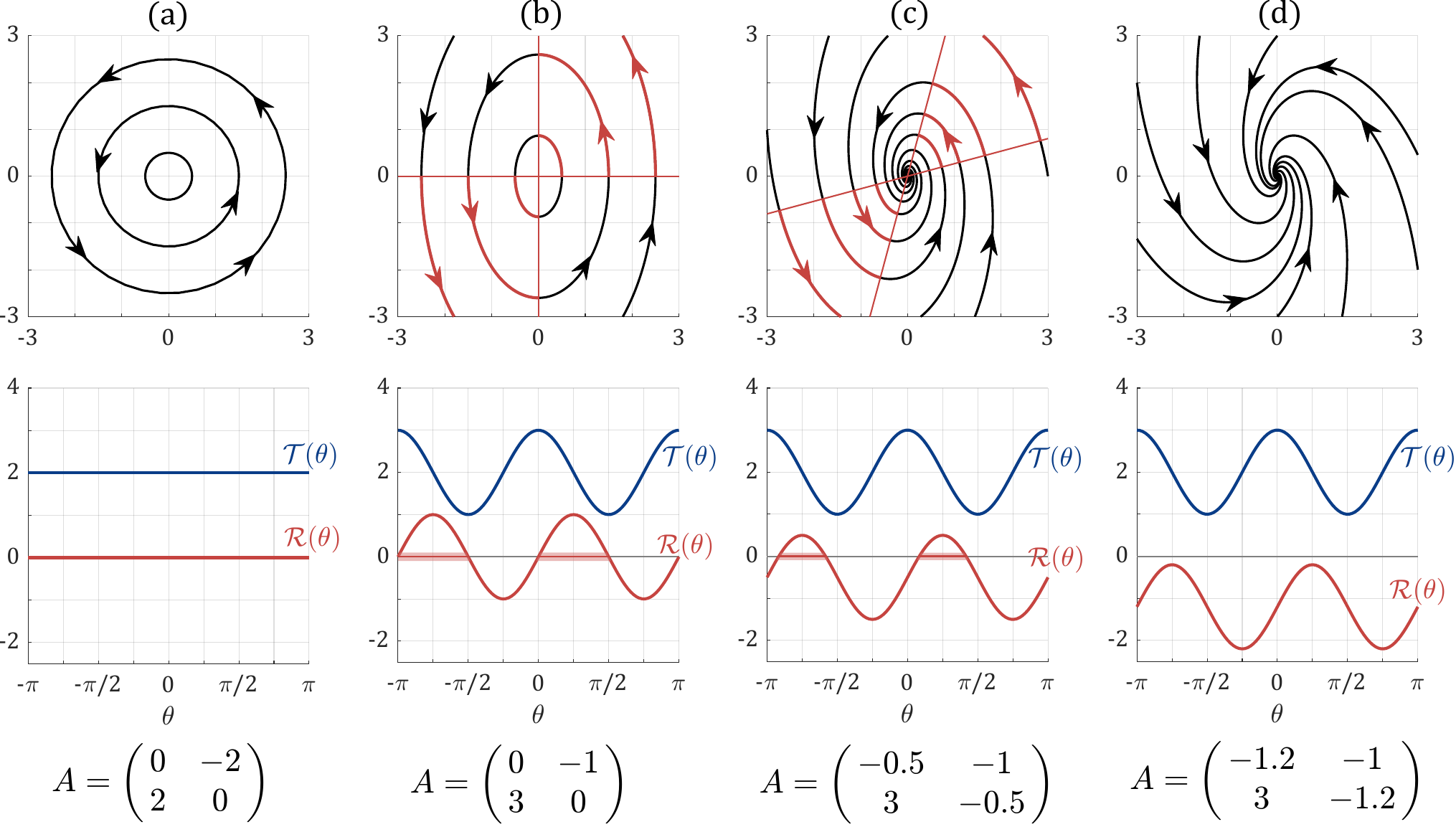}
    \setlength{\belowcaptionskip}{-4pt}
    \caption{
    Phase portraits of System \eqref{eqn:ODE}, and corresponding graphs of $\Rad$ and $\Tang$, associated with the matrices shown.
    In all four examples, $m_T=2$, and in (b)--(d), $\Tang$ is the same function.
     \textbf{(a)} $\Rad \equiv 0, \Tang \equiv 2$, and solutions to System \eqref{eqn:ODE} exhibit pure rotation with circular orbits, as in Corollary \ref{cor:ortho-char}$(iv)$.
     \textbf{(b)} 
     $m_R=0, p\ne 0$, and the orthovectors are orthogonal as in Proposition \ref{prop:balance} and Lemma \ref{lem:balance}. Since $\Tang>0$ in this example, the origin is a center, the orthovalues are positive, and solutions rotate counter-clockwise with elliptical orbits.
     \textbf{(c)} 
     $\Rad$ is shifted down from (b), so $m_R<0$ and the orthovectors are no longer orthogonal. Radial decay outweighs radial growth and the origin is a reactive spiral sink.
     \textbf{(d)} 
     When $\Rad$ is shifted down further, there are no real orthovectors or orthovalues, the system is no longer reactive ($\rho_1<0$), and the origin is a non-reactive spiral sink.
    }
    \label{fig:mR-decrease-series4}
\end{figure}
\textbf{Orthogonal orthovectors balance reactivity and attenuation.}~
The relationship between the 
radial and tangential dynamics of the
systems $\frac{dX}{dt}=AX$ and 
$\frac{dX}{dt}=-JAX$
described in the proof of Theorem \ref{theo:orthostructure} generalizes the well-known
connection between two-dimensional Hamiltonian and gradient systems: if the system $\frac{dX}{dt}=AX$ is Hamiltonian, then the orthogonal systems 
$\frac{dX}{dt}=JAX$ and $\frac{dX}{dt}=-JAX$
will be gradient systems 
\cite{hirsch2012differential}. The exchange between tangential and radial dynamics, which we highlight 
in Theorems \ref{theo:ortho-dualcalc} and \ref{theo:orthostructure},
is analogous to the exchange
of trajectories that are tangent to level sets in a Hamiltonian system with
trajectories that are orthogonal to level sets in a gradient system. 
\par
For System \eqref{eqn:ODE} to be a gradient system, $A$ must be curl-free, equivalently symmetric, which excludes reactive attractors or attenuating repellers (Proposition \ref{prop:sym}). From the dual perspective, for a two-dimensional linear system to be Hamiltonian, it must be divergence-free; equivalently $\tr(A) = 0$. In Proposition \ref{prop:balance}, we use the $\Rad$ and $\Tang$ lens to show that, in this case, System \eqref{eqn:ODE} admits periodic orbits.
For intuition, we show in Lemma \ref{lem:balance} 
(analogous to Lemma \ref{lem:sym} for eigenvectors), 
that, provided $\Rad \not\equiv 0$, if $\tr(A) = 0$, then the orthovectors of $A$ are orthogonal. Figure \ref{fig:mR-decrease-series4}b illustrates how orthogonal orthovectors equally split $\R^2$ into reactive and attenuating regions, resulting in a perfect balance between radial growth and radial decay. 
\input{RadTang-Table-Notation}
\begin{prop} 
\label{prop:balance}
Given System \eqref{eqn:ODE}, if $\tr(A)=0$ and $\Tang \not\equiv 0$, then the origin is a center if and only if $\Tang$ is single-signed and a balanced saddle (in the sense that the eigenvalues have equal magnitude) if and only if $\Tang$ is double-signed (i.e.\ $\tau_1>0>\tau_2$). 
\end{prop}
\begin{proof}
By Equation \eqref{eqn:m_R}, $\tr(A)=0$ if and only if $m_R=0$. 
Thus, if $\tr(A)=0$ and $\Tang \not\equiv 0$, $A$ has eigenvalues $\lambda_{1,2}=\pm \sqrt{-\tau_1\tau_2}$ by Theorem \ref{theo:eigenstructure}. In this case,
\begingroup\allowdisplaybreaks
\begin{align*}
\hspace{-0.15in} \Tang \mbox{ is single-signed} & \Leftrightarrow \Tang \mbox{ has no zeroes and } -\tau_1\tau_2<0\\
& \Leftrightarrow A 
\mbox{ has imaginary eigenvalues } \pm i\sqrt{\tau_1\tau_2}\\
& \Leftrightarrow \mbox{ the origin is a center.}
\end{align*}
\endgroup
\begin{align*}
\hspace{-1.15in}\mbox{Similarly, } \Tang \mbox{ is double-signed} & \Leftrightarrow \tau_1>0>\tau_2, \mbox{ so that } -\tau_1\tau_2>0,\\
& \Leftrightarrow A 
\mbox{ has real eigenvalues } \pm \sqrt{-\tau_1\tau_2}\\
& \Leftrightarrow \mbox{ the origin is a symmetric saddle.}
\end{align*}
\end{proof}
\par
\begin{lem}
~ \label{lem:balance}
Given a real $2\times2$ matrix, $A$, with $\Rad, m_R, m_T, p,  \theta_R$ and $\tau_i$ as summarized in Table \ref{tab:notation},
if $\Rad \not\equiv 0$, the following statements are equivalent:
\begin{enumerate}[label=(\roman*),labelindent=0.9\parindent,labelsep=1pt,align=left,leftmargin=*]
\item $\tr(A)=0$,
\item $m_R=0$,
\item $A$ has orthogonal orthovectors,
\item $A$ has real orthovalues $\mu_1=\tau_1=m_T+p$ and $\mu_2=\tau_2=m_T-p$, and
\item $A$ has orthovectors $V_1$ and $V_2$
at angles $\theta_R - \frac{\pi}{4}$ and $\theta_R + \frac{\pi}{4}$, respectively.
\end{enumerate}
\end{lem}
\begin{proof}
The proof of Lemma \ref{lem:balance} is analogous to that of Lemma \ref{lem:sym}, with the roles of $\Rad$ and $\Tang$ interchanged.
\end{proof}
\par 
In Sections \ref{sec:RadTang}--\ref{sec:ortho}, we have developed the four cornerstones of our framework. We defined the radial and tangential functions, $\Rad$ and $\Tang$ (Theorem \ref{theo:RadTang} and Definition \ref{def:RT}), and showed how they capture the dynamics of System \eqref{eqn:ODE} (Corollary \ref{cor:RadTangS^1}). We then leveraged the sinusoidal structure of $\Rad$ and $\Tang$ to define the reactive and attenuating regions (Definition \ref{def:ReactiveRegion}) and to establish the orthostructure of the system (Theorem \ref{theo:orthostructure}). 
This allows us to classify useful features of the transient dynamics, provides a language to describe characteristics of the reactive region, and paints a picture of the way solution trajectories traverse the reactive region. 
\par 
In the next sections we build on the cornerstones. In Section \ref{sec:SF}, we define standard forms for a matrix that help to reveal the reactivity structure. In Section \ref{sec:BoundingReactivity}, we bound and quantify 
the extent to which transient repulsion can accumulate when System \eqref{eqn:ODE} has a reactive attractor, 
and in Section \ref{sec:nonaut}, we illustrate how to nonautonomously `surf' the reactivity to infinity.

%% file: RadTang-Table-Notation.tex
{
\renewcommand{\arraystretch}{1.09}
\begin{table}[t]
\centering
\setlength\extrarowheight{2pt}
\begin{tabular}{|c|l|c|}
\hline
\multicolumn{1}{|l|}{\textbf{Symbol}} & \multicolumn{1}{l|}{\textbf{Description}} & \multicolumn{1}{c|}{\textbf{Notes}}\\
\hline
$\Rad$ & Radial function & \multirow{2}{*}{$AX = \Rad(\theta)X + \Tang(\theta)X^\perp$}\\
\cline{1-2}
$\Tang$ & Tangential function & \\
\hline
$m_R$ & Midline of $\Rad$ & $m_R = \,$\raisebox{1pt}[0pt][0pt]{$\tfrac{1}{2}$}$(a_{11}+a_{22})$ \\
\hline
$m_T$ & Midline of $\Tang$ & $m_T = \,$\raisebox{1pt}[0pt][0pt]{$\tfrac{1}{2}$}$(a_{21}-a_{12})$\\
\hline 
$p$ & Amplitude of $\Rad$ and $\Tang$ & $p = \,$\raisebox{1pt}[0pt][0pt]{$\tfrac{1}{2}$}$\!\sqrt{(a_{11}\!\!-\!a_{22})^2\!+\!(a_{12}\!+\!a_{21})^2}$\\
\hline
$\theta_R$ & Location of maximum of $\Rad$ & $\theta_R = \,$\raisebox{1pt}[0pt][0pt]{$\tfrac{1}{2}$}$\arctan(a_{12}\!+\!a_{21}, a_{11}\!-\!a_{22})$\\
\hline
$\theta_T$ & Location of maximum of $\Tang$ & \multicolumn{1}{c|}{$\theta_T = \theta_R -\pi/4$} \\
\hline 
$\rho_1$ & Reactivity, maximum of $\Rad$ & $\rho_1 = m_R + p$\\
\hline
$\rho_2$ & Attenuation, minimum of $\Rad$ & $\rho_2 = m_R-p$\\
\hline
$\tau_1$ & Maximum of $\Tang$ & $\tau_1 = m_T + p$\\
\hline
$\tau_2$ & Minimum of $\Tang$ & $\tau_2 = m_T - p $\\
\hline
$\delta_R$ & Reactivity radius & 
                    $\Rad(\theta)>0$ on
                    $(\theta_R-\delta_R, \theta_R+\delta_R)$\\
\hline
$\delta_T$ & Eigenvector separation radius & 
                $\Tang(\theta)>0$ on
                $(\theta_T-\delta_T, \theta_T+\delta_T)$\\
\hline
$\theta_1, \theta_2$ & Eigenvector angles & 
                $\theta_1=(\theta_T+\delta_T)$,
                $\theta_2=(\theta_T-\delta_T)$\\
\hline
$\lambda_1 \geq \lambda_2$ & Eigenvalues & 
                $\lambda_1=\Rad(\theta_1)$, 
                $\lambda_2=\Rad(\theta_2)$\\
\hline
$p_R$ & Eigenvalue separation radius & 
                $\lambda_1=(m_R+p_R)$, 
                $\lambda_2=(m_R-p_R)$\\ 
\hline
$\phi_1, \phi_2$ & Orthovector angles & 
                    $\phi_1=(\theta_R-\delta_R)$, 
                    $\phi_2=(\theta_R+\delta_R)$\\
\hline
$\mu_1 \geq \mu_2$ & Orthovalues & 
                    $\mu_1=\Tang(\phi_1)$, 
                     $\mu_2=\Tang(\phi_2)$\\
\hline
$p_T$ & Orthovalue separation radius & 
                    $\mu_1=(m_T+p_T)$, 
                    $\mu_2=(m_T-p_T)$\\
\hline
$\rho_{max}$ & Maximal amplification & Definition \ref{def:MA}\\
\hline
\end{tabular}~\\[1pt]
\begin{tabular}{|p{13mm}|p{102mm}|}
\hline
\multirow{2}{*}{Desmos}
 & (1) \href{https://www.desmos.com/calculator/uzopvqxo8p}{https://www.desmos.com/calculator/uzopvqxo8p}  (input $A$)
\\ \cline{2-2}
& 
(2) \href{https://www.desmos.com/calculator/xo8evgzgwo}{https://www.desmos.com/calculator/xo8evgzgwo} (input $\Rad$, $\Tang$)
\\
\hline
\end{tabular}~\\[1em]
\setlength{\abovecaptionskip}{-1pt}
\setlength{\belowcaptionskip}{-2pt}
\caption{These definitions are found in Theorems \ref{theo:RadTang}, \ref{theo:eigenstructure}, and \ref{theo:orthostructure}, Corollary \ref{cor:sinusoidal}, and Definitions \ref{defn:deltas}, \ref{def:ortho}, and \ref{def:MA}. Methods to compute $\delta_R$ and $\delta_T$ are given in Corollary \ref{cor:delta_R_T}. For geometric intuition, see Figures \ref{fig:RTdef}, \ref{fig:ReactiveRegion} and \ref{fig:saddle}. See 
the linked Desmos pages to interactively plot and explore $\Rad$ and $\Tang$.}
\label{tab:notation}
\end{table}
}

%% file: RadTang-6-StandardForms.tex
\section{Standard matrix forms for analyzing transient dynamics}
\label{sec:SF}
In this section we define four standard forms for a matrix $A$ (Definition \ref{def:std_forms}) as alternatives to the Jordan canonical form. Alternatives are needed because the Jordan form, although fundamental for studying asymptotic dynamics, suffers from the problem of masking transient dynamics (\cite{higham1993stiffness} and Proposition \ref{prop:sym}). Each of the alternatives we define highlights a different aspect of the dynamics of System \eqref{eqn:ODE} by conveniently locating either the reactive region or the eigenvectors relative to the coordinate axes. In Theorem \ref{theo:std_forms} we show that, much like the Jordan form, the standard forms of $A$ are easy to write down (in terms of the parameters of $\Rad$ and $\Tang$ and the eigenvalues and orthovalues of $A$). However, in contrast with the Jordan form, these standard forms are obtained by conjugation with pure rotation, and therefore respect the transient behavior of System \eqref{eqn:ODE} as well as the asymptotic behavior.
\begin{defn}
\label{def:std_forms}
A real-valued $2\times 2$ matrix $B$, with $\Rad, \Tang, \rho_1, \tau_1, \theta_R, \theta_T, \delta_R,$ and $\delta_T$ as summarized in Table \ref{tab:notation}, is in
\begin{enumerate}[label=(\roman*),labelindent=0.7\parindent,labelsep=2pt,align=left,leftmargin=*]
    \item
    $\Rad$-centered form if $\Rad_B(0)=\rho_1$, so $\Rad_B$ attains its maximum at $\theta_R=0$.
    \item
    $\Tang$-centered form if $\Tang_B(0)=\tau_1$, so $\Tang_B$ attains its maximum at $\theta_T=0$.
    \item
    $\Rad$-zeroed form if $\Rad_B(0)=0$ and $\Rad_B'(0)\geq 0$, so $\Rad_B$ attains its maximum at $\theta_R = \delta_R$.
    \item
    $\Tang$-zeroed form if $\Tang_B(0)=0$ and $\Tang_B'(0)\geq 0$, so $\Tang_B$ attains its maximum at $\theta_T = \delta_T$.
\end{enumerate}
\end{defn}
\par \medskip \noindent
\textbf{Dynamics of the standard forms.}
If $A$ is in $\Rad$-centered form, then $\theta_R=0$, and the reactivity of System  \eqref{eqn:ODE} is on the $x$-axis. If, moreover, $A$ has real orthovalues, then the orthovector angles (zeros of $\Rad_A$) are $\pm \delta_R$ and the reactive region is centered on the $x$-axis as in Figure \ref{fig:std_forms}a. If, by contrast, $A$ is in $\Rad$-zeroed form, then it has orthovector angles at $0$ and $2\delta_R$, so the $x$-axis forms one of the boundary lines of the reactive region. If the orthovalues are distinct ($\delta_R \neq 0$), then the condition $\Rad_A'(0)\geq 0$ ensures that the reactive region extends from the $x$-axis into the first and third quadrants as in Figure \ref{fig:std_forms}c. 
\par 
If $A$ is in $\Tang$-centered form, then $\theta_T=0$, and the maximum angular velocity of System  \eqref{eqn:ODE} occurs on the $x$-axis. If, moreover, $A$ has real eigenvalues, then the eigenvector angles (zeros of $\Tang_A$) are at $\pm \delta_T$ and the eigenvectors are symmetric\-al\-ly located relative to the coord\-inate axes as in Figure \ref{fig:std_forms}b. If, by contrast, $A$ is in $\Tang$-zeroed form, then it has eigen\-vector angles at $0$ and $2\delta_T$, so the $x$-axis is an eigen\-line. If the eigen\-values are distinct ($\delta_T \neq 0$), then the condition $\Tang_A'(0)\geq 0$ ensures that $\Tang_A > 0$ on $(0,2\delta_T)$ and, in the angular dynamics, traject\-ories are directed away from the $x$-axis, as in Figure \ref{fig:std_forms}d. 
\begin{figure}
    \centering
    \includegraphics[width=\textwidth]{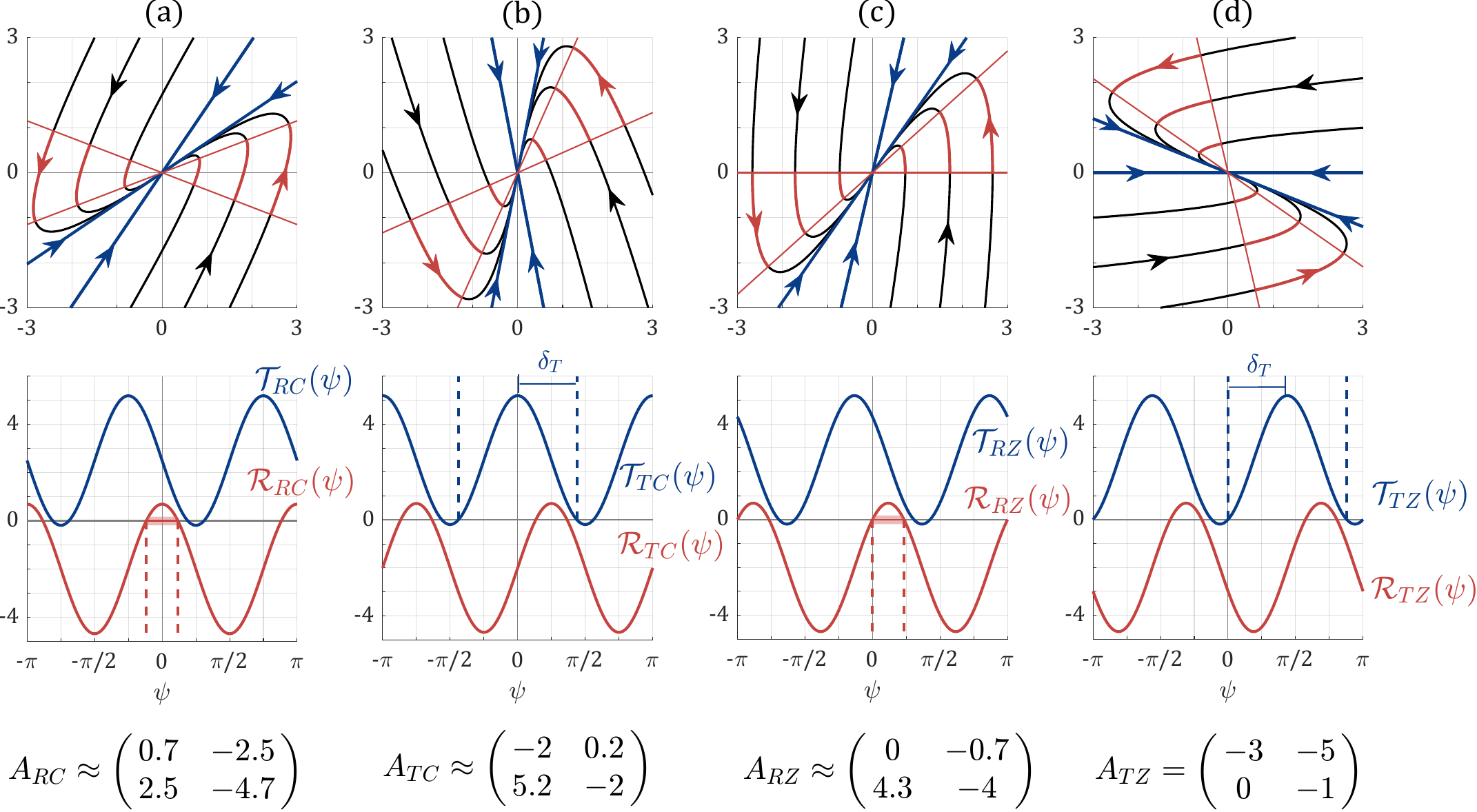}
    \setlength{\abovecaptionskip}{6pt}
    \caption{
    Phase portraits of System \eqref{eqn:ODE}, and corresponding graphs of $\Rad$ and $\Tang$, associated with the four standard forms of the matrix $A=\left(\begin{smallmatrix} -1 & -5 \\ 0 & -3 \end{smallmatrix}\right)$, illustrating 
    Theorem \ref{theo:std_forms}, Table \ref{tab:std_forms}, and Definition \ref{def:std_forms}. The original system associated with $A$ is shown in Figure \ref{fig:reac-increase-series3}b. 
    \textbf{(a)} $\Rad$-centered form, in which the reactive region (red) is centered on the $x$-axis, and $\Rad$ achieves its maximum at $\psi=0$.
    \textbf{(b)} $\Tang$-centered form, in which the eigendirections (dark blue) are symmetrically located relative to the coordinate axes, and $\Tang$ achieves its maximum at $\psi=0$.
    \textbf{(c)} $\Rad$-zeroed form, in which the $x$-axis is a boundary line of the reactive region (red), and $\Rad$ achieves its maximum at $\psi=\delta_R$.
    \textbf{(d)} $\Tang$-zeroed form, in which the $x$-axis is an eigendirection (dark blue), and $\Tang$ achieves its maximum at $\psi=\delta_T$.
    The four phase portraits can be viewed as rotations of each other, relative to the coordinate axes, while the radial and tangential functions are horizontal translations of each other, illustrating Lemma \ref{lem:horizontal_shift} and Corollary \ref{cor:conj-shift}.
    }
    \label{fig:std_forms}
\end{figure}
\par \medskip \noindent 
\textbf{Calculating the standard forms.}
After fixing some notation, we describe how to calculate the four standard forms for a given matrix $A$ (Theorem \ref{theo:std_forms}). Figure \ref{fig:std_forms} shows the standard forms and their corresponding dynamics for 
the system illustrated in
Figure \ref{fig:reac-increase-series3}b.
Table \ref{tab:std_forms} summarizes the standard forms, together with their associated 
radial and tangential functions, 
eigenvector angles and eigenvalues (when real), orthovector angles and orthovalues (when real), and conjugation angle $\gamma$.
\begin{notn}
\label{notn:rotation-matrix}
The matrix representing counter-clockwise rotation by $\gamma$ is denoted
\begin{equation*}
M_\gamma = \begin{pmatrix}\cos \gamma &  -\sin \gamma\\ \sin \gamma & \cos \gamma\end{pmatrix}
\end{equation*}   
\end{notn}
\begin{theo}
\label{theo:std_forms}
Let $A$ be a real $2 \times 2$ matrix, with $m_R$, $m_T$, $\rho_{i}$, $\tau_{i}$, $\lambda_{i}$, and $\mu_{i}$ as summarized in Table \ref{tab:notation}.
\begin{enumerate}[label=(\roman*),labelindent=0.7\parindent,labelsep=2pt,align=left,leftmargin=*]
\item 
If $\gamma=\theta_R$ and $A_{RC}=M_\gamma^{-1}AM_\gamma$, then $A_{RC}$ is in $\Rad$-centered form, and 
$$A_{RC}=\begin{pmatrix} \rho_1 & -m_T \\ m_T & \rho_2 \end{pmatrix}.$$
\item If $\gamma=\theta_T$ and $A_{TC}=M_\gamma^{-1}AM_\gamma$, then $A_{TC}$ is in $\Tang$-centered form, and 
$$A_{TC}=\begin{pmatrix} m_R & -\tau_2 \\ \tau_1 & m_R \end{pmatrix}.$$
\item When $A$ has real orthovalues $\mu_1 > \mu_2$ 
with corresponding reactivity radius $\delta_R$,
if $\gamma=\theta_R-\delta_R$ and $A_{RZ}=M_\gamma^{-1}AM_\gamma$, then $A_{RZ}$ is in $\Rad$-zeroed form, and 
$$A_{RZ}=\begin{pmatrix} 0 & -\mu_2 \\ \mu_1 & 2m_R \end{pmatrix}.$$
\item When $A$ has real eigenvalues $\lambda_1 > \lambda_2$ with corresponding eigenvector separation radius $\delta_T$, 
if $\gamma=\theta_T-\delta_T$ and $A_{TZ}=M_\gamma^{-1}AM_\gamma$, then $A_{TZ}$ is in $\Tang$-zeroed form, and 
$$A_{TZ}=\begin{pmatrix} \lambda_2 & -2m_T \\ 0 & \lambda_1 \end{pmatrix}.$$
\end{enumerate}
\end{theo}
\par \noindent
\begin{notn}
\label{notn:subscripts}
Henceforth, the $\Rad$-centered, $\Tang$-centered, $\Rad$-zeroed, and $\Tang$-zeroed forms of a matrix, and their corresponding radial and tangential functions, are denoted by the subscripts $RC$, $TC$, $RZ$, and $TZ$, respectively, as in Theorem \ref{theo:std_forms}. 
\end{notn}
\par \medskip \noindent
\textbf{Conjugation by rotation as a horizontal shift of $\Rad$ and $\Tang$.}
Before proving Theorem \ref{theo:std_forms}, 
we need to establish the impact of conjugation by rotation on the radial and tangential functions of System \eqref{eqn:ODE}. There are many ways to view the impact of conjugation. In polar coordinates, conjugation by $M_\gamma$ corresponds to a change of coordinates from $(r,\theta)$ to $(r, \psi)$, where $\psi = \theta - \gamma$. In the top row of Figure \ref{fig:std_forms} we represent the impact of conjugation as a rotation of the phase portrait by $\gamma$, clockwise, relative to the coordinate axes. Alternatively, we could view it as a counter-clockwise rotation of the coordinate axes relative to the phase portrait. In the bottom row of Figure \ref{fig:std_forms} we give the algebraic result of  conjugation to the standard forms 
for a specific matrix $A$.
\par 
In Lemma \ref{lem:horizontal_shift} we confirm that conjugation by rotation induces a horizontal shift of the radial and tangential functions, as observed in the middle row of Figure \ref{fig:std_forms}. This shifts the eigenvector and orthovector angles, and $\theta_R$ and $\theta_T$ accordingly, 
while preserving the eigenvalues, orthovalues, and other parameters of $\Rad$ and $\Tang$ (Corollary \ref{cor:conj-shift}).
\begin{lem}
\label{lem:horizontal_shift}
Given a real $2\times2$ matrix $A$, if $B=M_\gamma^{-1}AM_\gamma$, then 
\begin{equation*}
 \Rad_B(\theta) = \Rad_A(\theta + \gamma) \mbox{ and } \ 
 \Tang_B(\theta) = \Tang_A(\theta + \gamma).   
\end{equation*}
\end{lem}
\begin{proof}
If 
$X = \begin{pmatrix}r\cos \theta\\ r\sin\theta \end{pmatrix}$, 
then rotation by $\gamma$ gives
$M_\gamma X = \begin{pmatrix} r\cos(\theta + \gamma) \\ r\sin(\theta + \gamma)\end{pmatrix}.$
Therefore, by Theorem \ref{theo:RadTang},
\begin{equation*}
A (M_\gamma X) = \Rad_A(\theta + \gamma) (M_\gamma X) + \Tang_A(\theta + \gamma) (M_\gamma X)^\perp.
\end{equation*}
Since $M_\gamma$ represents rotation around the origin, $(M_\gamma X)^\perp = M_\gamma X^\perp$.  Thus
\begin{align*}
BX & =M_\gamma^{-1} A M_\gamma X \\
& = M_\gamma^{-1} (\Rad_A(\theta + \gamma) M_\gamma X + \Tang_A(\theta + \gamma) M_\gamma X^\perp)\\
& = \Rad_A(\theta + \gamma) M_\gamma^{-1} M_\gamma X + \Tang_A(\theta + \gamma) M_\gamma^{-1} M_\gamma X^\perp\\
& = \Rad_A(\theta + \gamma)X + \Tang_A(\theta + \gamma)X^\perp.
\end{align*}
By Theorem \ref{theo:RadTang}, we also have 
\begin{equation*}
BX = \Rad_B(\theta)X + \Tang_B(\theta)X^\perp,
\end{equation*}
and the lemma follows.
\end{proof}
\begin{cor}
\label{cor:conj-shift}
Let $A$ be a real $2 \times 2$ matrix, 
and let $\Rad$, $\Tang$, $m_R$, $m_T$, $p$, $\delta_R$, $\delta_T$, $\rho_{i}$, $\tau_{i}$, $\lambda_{i}$, and $\mu_{i}$ be as summarized in Table \ref{tab:notation}.
\begin{enumerate}[label=(\roman*),labelindent=0.7\parindent,align=left,labelsep=2pt,leftmargin=*]
\item The values of $m_R$, $m_T$, $p$, $\delta_R$, $\delta_T$, $\rho_{i}$, $\tau_{i}$, $\lambda_{i}$, and $\mu_{i}$ are invariant under conjugation of $A$ by $M_\gamma$.
\item If $B=M_\gamma^{-1}AM_\gamma$ and $A$ has eigenvector angles $\theta_i$ and orthovector angles $\phi_i$, then $B$ has eigenvector angles $\theta_i-\gamma$ and orthovector angles $\phi_i-\gamma$.
\item 
If $B=M_\gamma^{-1}AM_\gamma$, $\Rad_A$ achieves its maximum at angle $\theta_R$, and $\Tang_A$ achieves its maximum at angle $\theta_T$, then $\Rad_B$ achieves its maximum at angle $\theta_R-\gamma$ and $\Tang_B$ achieves its maximum at angle $\theta_T-\gamma$.
\end{enumerate}
\end{cor}
\begin{proof}
Part $(i)$ follows immediately from Lemma \ref{lem:horizontal_shift}.
For part $(ii)$, by Theorem \ref{theo:eigenstructure}, $A$ has an eigenvector located at angle $\theta_i$ if and only if $\Tang_A(\theta_i)=0$. 
Thus, by Lemma \ref{lem:horizontal_shift}, $\Tang_B(\theta_i - \gamma) = \Tang_A(\theta_i)=0$, so that $B$ has an eigenvector located at angle $\theta_i-\gamma$. The rest of part $(ii)$ and part $(iii)$ follow similarly from Theorem \ref{theo:orthostructure} and Corollary \ref{cor:sinusoidal}.
\end{proof}
\par \noindent
\textbf{Proof of the standard forms.}
We now prove Theorem \ref{theo:std_forms}, together with the details given in Table \ref{tab:std_forms}.
\begin{proof}[Proof of Theorem \ref{theo:std_forms}]
\textit{Part $(i)$, $\Rad$-centered form.}
If $\gamma=\theta_R$, then conjugation by $M_\gamma$ corresponds to a change in polar coordinates from $(r,\theta)$ to $(r,\psi)$, where $\psi=\theta-\theta_R$. By Corollary \ref{cor:conj-shift}, since $\Rad_A$ achieves its maximum at $\theta_R$, $\Rad_{RC}$ achieves its maximum at $\psi_R= \theta_R - \gamma = 0$, and so $A_{RC}$ is in $\Rad$-centered form (Definition \ref{def:std_forms}).
\par
To find $\Rad_{RC}$ and $\Tang_{RC}$ in the new $\psi$ coordinates, we apply Theorem \ref{theo:RadTang}, noting that $\psi_R=0$, and that, by Corollary \ref{cor:conj-shift}, $m_T, m_R$ and $p$ are all invariant under conjugation by $M_\gamma$, so can be calculated directly from $A$. Thus,
\begin{align}
\label{eqn:Rad-psi}
\Rad_{RC}(\psi) & = m_R + p \cos(2(\psi - 0)) = m_R + p \cos(2\psi)\\
\label{eqn:Tang-psi}
\Tang_{RC}(\psi) & = m_T - p \sin(2(\psi - 0)) = m_T - p \sin(2\psi).
\end{align}    
\par 
To find $A_{RC}$, we apply Proposition \ref{prop:RadTang-to-A} to Equations \eqref{eqn:Rad-psi} and \eqref{eqn:Tang-psi}, noting that, by Corollary \ref{cor:conj-shift} again, the reactivity $\rho_1=m_R+p$ and attenuation $\rho_2=m_R-p$ are invariant under conjugation by $M_\gamma$:
\begin{align*}
A_{RC} & = \begin{pmatrix} \Rad_{RC}(0) & -\Tang_{RC}(\pi/2) \\ \Tang_{RC}(0) & \Rad_{RC}(\pi/2) \end{pmatrix}
 = \begin{pmatrix} m_R+p\cos(0) & -m_T+p\sin(\pi) \\ m_T-p\sin(0) & m_R+p\cos(\pi) \end{pmatrix}\\
& = \begin{pmatrix} m_R+p & -m_T \\ m_T & m_R-p \end{pmatrix}
= \begin{pmatrix} \rho_1 & -m_T \\ m_T & \rho_2 \end{pmatrix}.
\end{align*} 
\par 
The proof of part $(ii)$, for $\Tang$-centered form, is similar to that of part $(i)$ after noting that $\theta_T = \theta_R - \pi/4$ (Corollary \ref{cor:sinusoidal}), and so $\psi_T=0 \Leftrightarrow \psi_R = \pi/4$ 
\par 
\textit{Part $(iii)$, $\Rad$-zeroed form.}
If $\gamma=\theta_R-\delta_R$, then conjugation by $M_\gamma$ corresponds to changing coordinates from $(r,\theta)$ to $(r,\psi)$, where $\psi=\theta-(\theta_R-\delta_R)$.
By Corollary \ref{cor:conj-shift}, since $\Rad_A$ achieves its maximum at $\theta_R$, $\Rad_{RZ}$ achieves its maximum at $\psi_R=\theta_R-(\theta_R-\delta_R)=\delta_R$.
We know, from Theorem \ref{theo:orthostructure},
that the zeros of $\Rad_A$ are at $\theta_R \pm \delta_R$, and $\Rad_A(\theta)>0$ on $(\theta_R-\delta_R, \theta_R+\delta_R)$.   
Thus, by Corollary \ref{cor:conj-shift} again, the zeroes of $\Rad_{RZ}$ are at $\psi=0$ and $\psi = 2\delta_R$, and $\Rad_{RZ}(\psi)>0$ on  $(0,2\delta_R)$. Hence, $A_{RZ}$ is in $\Rad$-zeroed form.  
Applying Theorem \ref{theo:RadTang}: 
\begin{align*}
\Rad_{RZ}(\psi) & = m_R + p \cos(2(\psi - \delta_R)) \\
\Tang_{RZ}(\psi) & = m_T - p \sin(2(\psi - \delta_R)). 
\end{align*} 

To apply Proposition \ref{prop:RadTang-to-A} to find $A_{RZ}$, note that, by Theorem \ref{theo:orthostructure}, 
$\psi=0$ is the orthovector angle associated with the larger orthovalue, and so the first column of $A_{RZ}$ is given by 
\[
\Rad_{RZ}(0) = 0 \  \mbox{ and } \ 
\Tang_{RZ}(0) = \mu_1. \\
\]
To find $\Rad_{RZ}(\pi/2)$, note that 
\begin{align}
\Rad_{RZ}(0) = 0 
& \implies m_R + p \cos(-2\delta_R) = 0 \notag \\
& \implies  p \cos(-2\delta_R) = -m_R, \label{eq:cosdelta_R_prep}
\end{align}
\vspace{-1em}
and hence
\begin{align*}
\Rad_{RZ}(\pi/2) 
& = m_R + p \cos(\pi - 2\delta_R) \\
& = m_R - p \cos(-2\delta_R) = 2m_R.
\end{align*}
Similarly, to find $\Tang_{RZ}(\pi/2)$, note that 
\begin{align}
\Tang_{RZ}(0) = \mu_1 = m_T+p_T 
& \implies m_T-p\sin(-2\delta_R) = m_T+p_T \notag \\
& \implies p\sin(-2\delta_R) =-p_T, \label{eq:sindelta_R_prep}
\end{align}
\vspace{-1em}
and hence
\begin{align*}
\Tang_{RZ}(\pi/2) & = m_T-p\sin(\pi - 2\delta_R) \\
& = m_T+p\sin(-2\delta_R) = m_T-p_T = \mu_2.
\end{align*}
Thus,
\begin{align*}
A_{RZ} & = \begin{pmatrix} \Rad_{RZ}(0) & -\Tang_{RZ}(\pi/2) \\ \Tang_{RZ}(0) & \Rad_{RZ}(\pi/2) \end{pmatrix}
 = \begin{pmatrix} 0 & -\mu_2 \\ \mu_1 & 2m_R \end{pmatrix}.
\end{align*}
\par 
The proof of part $(iv)$, for $\Tang$-zeroed form, is similar to that of part $(iii)$ after noting that  $\psi_T=\delta_T \Leftrightarrow \psi_R = \delta_T+\pi/4$, and that in $\Tang$-zeroed form, $\psi=0$ is the eigenvector angle associated with the smaller eigenvalue $\lambda_2$ (Theorem \ref{theo:eigenstructure}).
\end{proof}
\input{RadTang-Table-SF}
\textbf{Calculating $\delta_R$ and $\delta_T$.}
One immediate application of Theorem \ref{theo:std_forms} is that 
methods for calculating the reactivity radius, $\delta_R$, and the eigenvector separation radius, $\delta_T$, are easily developed using the $\Rad$-centered and $\Tang$-centered forms, respectively.
\begin{cor}
\label{cor:delta_R_T}
Let $A$ be a real, $2\times2$ matrix and let 
$m_R, m_T, p, p_R, p_T, \delta_R$ and $\delta_T$ be as summarized in Table \ref{tab:notation}. 
\begin{enumerate}[label=(\roman*),labelindent=0.7\parindent,align=left,labelsep=2pt,leftmargin=*]
\item If $A$ has real orthovalues, then 
\begin{equation}
\label{eq:deltaR_calc}
\cos(2\delta_R) = \frac{-m_R}{p} \quad \text{and} \quad \sin(2\delta_R) = \frac{p_T}{p}.
\end{equation}
\item If $A$ has real eigenvalues, then 
\begin{equation}\label{eq:deltaT_calc}
    \cos(2\delta_T) = \frac{-m_T}{p} \quad \text{and} \quad \sin(2\delta_T) = \frac{p_R}{p}.
\end{equation}
\end{enumerate}
\end{cor}
\begin{proof}
For part $(i)$, we use Theorem \ref{theo:std_forms} and Corollary \ref{cor:conj-shift} to assume, without loss of generality, that $A$ is in $\Rad$-centered form.
Then 
\[
\Rad_{RC}(\psi) = m_R + p\cos(2\psi) \ \mbox{ and } \ 
\Tang_{RC}(\psi) = m_T - p\sin(2\psi),
\]
and the orthovectors (zeros of $\Rad_{RC}$) are at $\pm\delta_R$ (see Table \ref{tab:std_forms}). Thus, 
\begin{equation*}
 \Rad_{RC}(\delta_R) = m_R+p\cos(2\delta_R) = 0,  
\end{equation*}
and the first equation of  
\eqref{eq:deltaR_calc} follows. 
By Theorem \ref{theo:orthostructure}, the orthovalue $\mu_2 = m_T-p_T$ is given by $\Tang_{RC}(\delta_R)$, and so
\begin{equation*}
\Tang_{RC}(\delta_R) = m_T - p\sin(2\delta_R) = m_T - p_T,
\end{equation*} 
yielding the second equation of \eqref{eq:deltaR_calc}. While these equations for $\delta_R$ follow most naturally from $\Rad$-centered form, they also follow from $\Rad$-zeroed form using Equations \eqref{eq:cosdelta_R_prep} and \eqref{eq:sindelta_R_prep} within the proof of Theorem \ref{theo:std_forms}.
\par
The proof of part $(ii)$ is analogous to that of part $(i)$, with the roles of $\Rad$ and $\Tang$ interchanged. This time we assume, without loss of generality, that $A$ is in $\Tang$-centered form. Then 
\[
\Rad_{TC}(\psi)=m_R + p\sin(2\psi) \ \mbox{ and } \ \Tang_{TC}(\psi)=m_T + p\cos(2\psi),
\]
and the eigenvectors (zeros of $\Tang_{TC}$) are at $\pm\delta_T$. Thus, 
\begin{equation*}
 \Tang_{TC}(\delta_T) = m_T+p\cos(2\delta_T) = 0,   
\end{equation*}
yielding the first equation of \eqref{eq:deltaT_calc}. 
By Theorem \ref{theo:eigenstructure}, the eigenvalue $\lambda_1 = m_R+p_R$ is given by $\Rad_{RC}(\delta_T)$, and so
\begin{equation*}
\Rad_{RC}(\delta_T) = m_R + p\sin(2\delta_T) = m_R - p_R,
\end{equation*} yielding the second equation of \eqref{eq:deltaT_calc}.
\end{proof}
\par 
\Remark
Corollary \ref{cor:delta_R_T} shows that, up to a phase shift, $\Rad$ is determined by $\delta_R$ and $p_T$, the orthovector and orthovalue separation radii, respectively. Similarly, $\Tang$ is determined, up to a phase shift, by $\delta_T$ and $p_R$, the eigenvector and eigenvalue separation radii, respectively.
\par \medskip 
The value of matrices in $\Tang$-zeroed and $\Tang$-centered forms for studying reactivity has long been recognized in the literature. For example, in their first illustration of reactivity, Neubert and Caswell \cite{neubert1997alternatives} use an upper triangular matrix, which places an eigenvector on the $x$-axis.
Josi{\'c} and Rosenbaum \cite{josic2008unstable} use a similar upper triangular matrix and a $\Tang$-centered form in their study of the way transient reactivity can accumulate to yield instability in nonautonomous systems. 
The $\Tang$-centered form is also used by Harrington \textit{et.\ al.} \cite{harrington2022reactivity} to show that reactivity can lead to arbitrarily large growth of perturbations in the $\ell_1$ norm, and by Mierczy{\'n}ski \cite{mierczynski2017instability} in a non-autonomously switching linear system to show that reactivity can lead to instability.
\par
Vesipa and Ridolfi \cite{vesipa2017impact} use examples in a variant of $\Rad$-zeroed form, with an orthovector on the $y$-axis rather than the $x$-axis, (e.g.\ $A_R$ in \cite{vesipa2017impact}), but their analysis is based on the eigenstructure, not the orthostructure. We have yet to find a reference to the orthostructure based value of the $\Rad$-centered and $\Rad$-zeroed forms in the literature.
In the next section, we use these $\Rad$-focused forms to quantify the accumulated radial growth of perturbations in System \eqref{eqn:ODE}.

%% file: RadTang-Table-SF.tex
	\begin{landscape}
	\begin{table}[t]
		\centering\small
	
		\begin{tabular}{|r|c|c|c|c|}
			\hline
			&&&&\\[-3.25pt]
			 & \textbf{$\Rad$-centered} & \textbf{$\Tang$-centered} & \textbf{$\Rad$-zeroed} & \textbf{$\Tang$-zeroed}\\
			 &&&&\\[-3.25pt]
			\hline
			&&&&\\[-3pt]
			Matrix: &
			$A_{RC} = \begin{pmatrix}\rho_1 & -m_T\\ m_T & \rho_2\end{pmatrix}$ &
			$A_{TC} = \begin{pmatrix}m_R & -\tau_2\\ \tau_1 & m_R\end{pmatrix}$  &
			$A_{RZ} = \begin{pmatrix} 0 & -\mu_2\\\mu_1&2m_R\end{pmatrix}$
			& 
			$A_{TZ} = \begin{pmatrix}\lambda_2 & -2m_T\\ 0 & \lambda_1\end{pmatrix}$\\
			&&&&\\[-3pt]
			\hline
             &&&&\\[-3pt]
			$\begin{aligned}[t] \Rad(\psi)= &\\ \Tang(\psi)=&\end{aligned}$&
			$\begin{aligned}[t] & m_R+p\cos(2\psi)\\ 
				&m_T -p\sin(2\psi)\end{aligned}$ &
			$\begin{aligned}[t] & m_R+p\sin(2\psi)\\ &m_T + p\cos(2\psi)\end{aligned}$  &
			$\begin{aligned}[t] & m_R+p\cos(2(\psi-\delta_R))\\ &m_T - p\sin(2(\psi-\delta_R))\end{aligned}$ 
			&  
			$\begin{aligned}[t] & m_R+p\sin(2(\psi-\delta_T))\\ 
				& m_T + p\cos(2(\psi-\delta_T))\end{aligned}$ \\
			&&&&\\[-3pt]
            \hline 
            &&&&\\[-3pt]
            Maxima: &
            $\psi_R=0$, \,\, $\psi_T = -\tfrac{\pi}{4}$&
            $\psi_R=\tfrac{\pi}{4}$,\,\,$\psi_T = 0$&
            $\psi_R=\delta_R$,\,\,$\psi_T = \delta_R-\tfrac{\pi}{4}$&
            $\psi_R=\delta_T+\tfrac{\pi}{4}$,\,\, $\psi_T = \delta_T$\\
			&&&&\\[-3pt]
            \hline
			&&&&\\[-3pt]
			Horizontal shift: &
			$\gamma = \theta_R$ &
			$\gamma =  \theta_T = \theta_R - \pi/4$ &
			$\gamma = \theta_R-\delta_R$ &
			$\gamma = \theta_T-\delta_T$ \\
			&&&&\\[-4pt]
			\hline 
            &&&&\\[-3pt]
			Eigenvectors:& 
			$\tfrac{-\pi}{\,\,4}\pm \delta_T$  &  
			$\pm\delta_T$   &
			$\delta_R-\tfrac{\pi}{4}\pm\delta_T$
			& 
			$0, \,2\delta_T$ \\
			&&&&\\[-3.2pt]
			\hline
			&&&&\\[-3pt]
			$\begin{aligned}[t] 
                                \lambda_1 = m_R+p_R\\
                                \lambda_2 = m_R-p_R
            \end{aligned}$&
            $\begin{aligned}[t] \lambda_1 &= \Rad_{RC}(\tfrac{-\pi}{\,\,4}+\delta_T)\\
                 \lambda_2 &= \Rad_{RC}(\tfrac{-\pi}{\,\,4}-\delta_T) \end{aligned}$ &
            $\begin{aligned}[t] \lambda_1 &= \Rad_{TC}(+\delta_T)\\
                 \lambda_2 &= \Rad_{TC}(-\delta_T) \end{aligned}$ &
            $\begin{aligned}[t] \lambda_1 &=\Rad_{RZ}(\delta_R-\tfrac{\pi}{4}+\delta_T)\\ 
				\lambda_2&=\Rad_{RZ}(\delta_R-\tfrac{\pi}{4}-\delta_T)\end{aligned}$ &
            $\begin{aligned}[t] \lambda_1 &= \Rad_{TZ}(2\delta_T)\\
                 \lambda_2 &= \Rad_{TZ}(0) \end{aligned}$ 
			\\
			&&&&\\[-3pt]
			\hline   
			&&&&\\[-3pt]
			Orthovectors:&   
			$\pm\delta_R$   & 
			$\tfrac{\pi}{4} \pm\delta_R$ &            
			$0, \, 2\delta_R$
			& 
			$\delta_T+\tfrac{\pi}{4} \pm\delta_R$ \\
			&&&&\\[-4pt]
			\hline
			&&&&\\[-3pt]
			$\begin{aligned}[t] 
                                \mu_1 = m_T+p_T\\
                                \mu_2 = m_T-p_T
            \end{aligned}$&
            $\begin{aligned}[t] \mu_1 &= \Tang_{RC}(- \delta_R)\\ 
				 \mu_2 &= \Tang_{RC}(+\delta_R)\end{aligned}$ &
            $\begin{aligned}[t] \mu_1 &= \Tang_{TC}(\tfrac{\pi}{4} - \delta_R)\\ 
				 \mu_2 &= \Tang_{TC}(\tfrac{\pi}{4} + \delta_R)\end{aligned}$ &
            $\begin{aligned}[t] \mu_1 &= \Tang_{RZ}(0)\\ 
				 \mu_2 &= \Tang_{RZ}(2\delta_R)\end{aligned}$ &
            $\begin{aligned}[t] \mu_1 &= \Tang_{TZ}(\delta_T+\tfrac{\pi}{4}- \delta_R)\\ 
				 \mu_2 &= \Tang_{TZ}(\delta_T+\tfrac{\pi}{4}+ \delta_R)\end{aligned}$ \\[-3pt]
			&&&&\\
			\hline 
		\end{tabular}
        \captionsetup{skip=8pt,width=\linewidth}
	\caption{
Each column characterizes one of the standard forms of a matrix, $A$, given by Theorem \ref{theo:std_forms}. The first two rows specify
 the matrix form and the corresponding $\Rad$ and $\Tang$ equations. The third row locates the maxima of $\Rad$ and $\Tang$.
 The fourth row specifies the rotation angle, $\gamma$, yielding the standard form via conjugation, $M_\gamma^{-1}AM_\gamma$. The remaining rows describe the eigenvectors, eigenvalues, orthovectors and orthovalues (when they are real). 
 The parameters $m_R$, $m_T$, $p$, $\tau_i$, $\rho_i$, $\delta_T$, $\delta_R$ and the eigenvalues $\lambda_i$ and orthovalues $\mu_i$ are invariant across all four standard forms (see Corollary \ref{cor:conj-shift} and Table \ref{tab:notation}). 
}
 \label{tab:std_forms}
\end{table}
	\end{landscape}

%% file: RadTang-7-max-amp.tex
\section{Quantifying Maximal Amplification}
\label{sec:BoundingReactivity}
In this section, we contrast properties of the reactivity -- the maximal {\it instantaneous} rate of radial amplification in a system -- with those of the {\it cumulative} maximal amplification (Definition \ref{def:MA}).
It is well known that if $A$ is normal, then System \eqref{eqn:ODE} does not admit a reactive attractor (See Proposition \ref{prop:sym}, and the comments thereafter). For non-normal matrices with real eigenvalues, however, Theorems \ref{theo:GivenLambdas-AnyReactivity} and \ref{theo:GivenEigenvectors-AnyReactivity} show 
that the reactivity of an attractor is independent of the eigenvalues and eigenvectors.
At first, these theorems can be surprising, as they run counter to the widespread intuition that reactivity increases as eigenvectors converge (as in Figure \ref{fig:reac-increase-series3}). The intuition is restored by Theorem \ref{theo:MaxAmpUpperBound} and Corollary \ref{cor:MaxAmpUpperBound-delta_T}, where we give bounds on the maximal amplification of the system that do, indeed, depend on the eigenvector or orthovector separation of $A$. In Theorem \ref{theo:MaxAmpFormula}, we calculate the maximal amplification exactly.
Analogous results can be shown for attenuating repellers by viewing them as time-reversed reactive attractors.
\par 
Recall from the introduction that Neubert and Caswell \cite{neubert1997alternatives} introduce the
maximal amplification of perturbations in a reactive attracting system as follows:
\begin{defn}
\label{def:MA}
If System \eqref{eqn:ODE} has a reactive attractor, the maximal amplification, $\rho_{max}$, is the maximum factor by which a perturbation can be amplified before returning to the origin:
\begin{equation}
\label{eqn:rho_max_def}
    \rho_{max} = \max_{X_0 \neq 0, t>0} \left(\frac{r(t)}{r_0}\right) 
    = \max_{X_0 \in S^1, t>0} \left( r(t)\right),
\end{equation}
where $X(t)$ is the solution with initial condition $X(0), r(t)=\|X(t)\|$, and $r_0=\|X(0)\|$.
The amplification time, $t_{max}$, is the time required to achieve this maximal amplification:
\begin{equation*}
\rho_{max} = \max_{X_0 \in S^1}  \left( r(t_{max})\right).
\end{equation*}
\end{defn}
\par 
\Remark
The second equality in Equation \eqref{eqn:rho_max_def} follows from the linearity of System \eqref{eqn:ODE}.
\par \medskip \noindent 
{\bf Reactivity is unbounded.}
To prepare for Theorems \ref{theo:GivenLambdas-AnyReactivity} and \ref{theo:GivenEigenvectors-AnyReactivity}, we show, in Lemma \ref{lem:GivenDeltas-AnyReactivity}, that the reactivity of an attractor is independent of the eigenvector and orthovector separation radii, $\delta_R$ and $\delta_T$. Lemma \ref{lem:GivenDeltas-AnyReactivity} provides a direct contrast with Theorem \ref{theo:MaxAmpUpperBound} and Corollary \ref{cor:MaxAmpUpperBound-delta_T}, in which the maximal amplification, $\rho_{max}$, of a system, is bounded by functions of $\delta_R$ and $\delta_T$, respectively, and with Theorem \ref{theo:MaxAmpFormula}, where $\rho_{max}$ is shown to be a function of $\delta_R$ and $\delta_T$.
\begin{lem}
\label{lem:GivenDeltas-AnyReactivity}
Fix any $\delta_R\in (0,\pi/2),\ \delta_T\in [0,\pi/2)$, and $\rho > 0$. The linear system (System \eqref{eqn:ODE}) with coefficient matrix $A$ given by 
$$
A = \frac{\rho}{1-\cos(2\delta_R)}
\begin{pmatrix} -\cos(2\delta_R) & 1+\cos(2\delta_T) \\ 1-\cos(2\delta_T) & -\cos(2\delta_R) \end{pmatrix}
$$ 
has real eigenvalues and real orthovalues, eigenvector separation radius $\delta_T$, reactivity radius $\delta_R$, and reactivity $\rho$.\\[-6pt]
\end{lem}
\par 
\Remark
The matrix $A$ given in Lemma \ref{lem:GivenDeltas-AnyReactivity} is in $\Tang$-centered form. By Corollary \ref{cor:conj-shift}, any conjugation of $A$ by rotation also has the same eigenvector separation radius, reactivity radius, and reactivity as $A$.
\par
\begin{proof}[Proof of Lemma \ref{lem:GivenDeltas-AnyReactivity}]
It is straightforward to check that $A$ has real eigenvalues and orthovalues by Corollaries \ref{cor:eigenvalues} and \ref{cor:ortho-char},
since $|m_T|\leq p$, and $|m_R|<p$.
Then, by Corollary \ref{cor:delta_R_T}, the eigenvector separation radius of $A$ is given by:
\begin{align*}
    \frac{1}{2}\arccos\left(\frac{-m_T}{p}\right)
   & = \frac{1}{2}\arccos\left(\frac{-(a_{21}-a_{12})}{\sqrt{(a_{11}-a_{22})^2+(a_{12}+a_{21})^2}}\right)\\
   & = \frac{1}{2}\arccos\left(\frac{2\cos(2\delta_T)}{\sqrt{4}}\right) \\ 
   & = \delta_T
\end{align*}
Similarly, the reactivity radius of System \ref{eqn:ODE} is given by:
\begin{align*}
    \frac{1}{2}\arccos\left(\frac{-m_R}{p}\right)
   & = \frac{1}{2}\arccos\left(\frac{-(a_{11}+a_{22})}{\sqrt{(a_{11}-a_{22})^2+(a_{12}+a_{21})^2}}\right)\\
   & = \frac{1}{2}\arccos\left(\frac{2\cos(2\delta_R)}{\sqrt{4}}\right)\\ 
   & = \delta_R
\end{align*}
Finally, by Theorem \ref{theo:reactive-rho=rho_1}, the reactivity of System \ref{eqn:ODE} is given by:
\begin{align*}
    m_R+p & =
    \frac{1}{2}(a_{11}+a_{22})+\frac{1}{2}{\sqrt{(a_{11}-a_{22})^2+(a_{12}+a_{21})^2}}\\
   & = \frac{\rho}{1-\cos(2\delta_R)}\left(-\cos(2\delta_R)+1\right)\\
   & = \rho
\end{align*}
\end{proof}
\par 
\Remark
The matrix $A$ in Lemma \ref{lem:GivenDeltas-AnyReactivity} was constructed geometrically using the following steps. Begin by constructing radial and tangential functions in $\Tang$-centered form from the `base' functions $\sin(2\theta)$ and $\cos(2\theta)$. Adjust by applying vertical shifts of $-\cos(2\delta_R)$ and $-\cos(2\delta_T)$, respectively, to ensure that $\Rad(\tfrac{\pi}{4}\pm \delta_R)=0$ and $\Tang(\pm\delta_T)=0$. Then scale both functions by $p= \rho / (1-\cos(2\delta_R))$ to ensure that $\Rad(\theta)$ has the desired maximum, $\rho$. Finally, with these $\Rad$ and $\Tang$, apply Proposition \ref{prop:RadTang-to-A} to determine the coefficients of $A$.
\par 
\begin{theo}
\label{theo:GivenLambdas-AnyReactivity}
Given any real $\lambda_2 \leq \lambda_1 < 0$ and any $\rho > 0$, there is a real $2\times 2$ matrix $A$, with eigenvalues $\lambda_1$ and $\lambda_2$, for which System \ref{eqn:ODE} has an attractor with reactivity $\rho$.
\end{theo}
\begin{proof}
To construct a matrix $A$ with eigenvalues $\lambda_2 \leq \lambda_1 < 0$, note that, by Corollary \ref{cor:delta_R_T}, $\delta_R$ must be given by
\[\delta_R = \frac{1}{2}\arccos\left(\frac{-m_R}{p}\right)\]
where $m_R = \tfrac{1}{2}\left(\lambda_1+\lambda_2\right)<0$ (by assumption), and $p=\rho-m_R>0$ (by Theorem \ref{theo:reactive-rho=rho_1}), so that $m_R/p < 0$ and, therefore, $\delta_R\in(0,\pi/4)$.
\par
Similarly, by Corollary \ref{cor:delta_R_T}, $\delta_T$ is given by 
\[\delta_T = \frac{1}{2}\arcsin\left(\frac{p_R}{p}\right)\]
where $p_R = \lambda_1-m_R\ge0$, so that $p_R/p > 0$, and we can choose $\delta_T \in [0,\pi/4)$.
\par
Now we have $\delta_R, \delta_T$ and $\rho$, and so the result follows from Lemma \ref{lem:GivenDeltas-AnyReactivity}. 
\end{proof}
\par
\Remark
Once again, the matrix $A$ given by Lemma \ref{lem:GivenDeltas-AnyReactivity} in Theorem \ref{theo:GivenLambdas-AnyReactivity} is in $\Tang$-centered form, and, by Corollary \ref{cor:conj-shift}, any conjugation of $A$ by rotation will have the same eigenvalues, $\lambda_i$, and reactivity, $\rho$, as $A$.
\par 
\begin{theo}
\label{theo:GivenEigenvectors-AnyReactivity}
Given any non-orthogonal, linearly independent pair of vectors $V_1,V_2 \in \mathbb{R}^2$ and any $\rho > 0$, there is a real $2\times 2$ matrix $B$ with eigenvectors $V_1,V_2$, for which the corresponding System \eqref{eqn:ODE} has an attractor with reactivity $\rho$.
\end{theo}
\begin{proof}
Let $\theta_1, \theta_2 \in [0,2\pi)$ 
measure the angles from the $\theta=0$ axis to $V_1$ and $V_2$, respectively.
Without loss of generality, assume the vectors $V_1,V_2$ are labeled so that $\theta_1-\theta_2 \in (0,\pi/2) \cup (\pi/2,\pi)$, and 
let $\delta_T=\frac{1}{2}(\theta_1-\theta_2) \in (0,\pi/4)\cup (\pi/4,0)$. 
The two cases $\delta_T \in (0,\pi/4)$ and $\delta_T \in (\pi/4, \pi/2)$ correspond to the cases $m_T<0$ and $m_T>0$, respectively.
In either case, choose $\delta_R\in (0, |\delta_T-\pi/4|)$ to guarantee a reactive region that does not contain eigenvectors.
Now we have $\delta_R, \delta_T$ and $\rho$, and can apply Lemma \ref{lem:GivenDeltas-AnyReactivity}
to find a matrix, $A$, in $\Tang$-centered form with real eigenvalues, eigenvectors at angles $\pm \delta_T$, reactivity radius $\delta_R$, and reactivity $\rho$, which we will then conjugate with the appropriate rotation to place the eigenvectors at $V_1, V_2$.
\par
To confirm that the eigenvalues of $A$
are both negative, note that, by Corollary \ref{cor:delta_R_T}, 
\begin{align*}
    m_R & = -p\cos(2\delta_R) \mbox{ and } p_R = p\sin(2\delta_T).
\end{align*}
Therefore the largest eigenvalue, given by 
\begin{align*}
    m_R+p_R &= -p\big(\cos(2\delta_R)-\sin(2\delta_T)\big)\\
            &= -p\big(\cos(2\delta_R)-\cos(2(\delta_T-\pi/4)\big).
\end{align*}
Since $\delta_R\in (0,|\delta_T-\pi/4|)\subseteq (0,\pi/4)$, 
\begin{align*}
    \delta_R & < |\delta_T-\pi/4| \\
    & \implies \cos(2\delta_R) > \cos(2|\delta_T-\pi/4|) = \cos(2(\delta_T-\pi/4))\\ 
    & \implies m_R+p_R<0,
\end{align*}
and hence both eigenvalues are negative.

Finally, conjugating by rotation to place the eigenvectors at $V_1$ and $V_2$, as desired, without changing the eigenvalues or the reactivity, 
let $B=M_{\gamma}^{-1}AM_{\gamma}$, where $\gamma=-\frac{1}{2}(\theta_1+\theta_2)$.
Then, by Lemma \ref{lem:horizontal_shift}, 
\begin{align*}
    \Tang_B(\theta_1) & =\Tang_A\left(\theta_1-\tfrac{1}{2}(\theta_1+\theta_2)\right)
    =\Tang_A\left(\tfrac{1}{2}(\theta_1-\theta_2)\right)
    =\Tang_A(\delta_T)=0, \mbox{ and }\\
   \Tang_B(\theta_2) & =\Tang_A\left(\theta_2-\tfrac{1}{2}(\theta_1+\theta_2) \right)
    =\Tang_A\left(-\tfrac{1}{2}(\theta_1-\theta_2)\right)
    =\Tang_A(-\delta_T)=0.
\end{align*}
Thus System \ref{eqn:ODE}, with matrix $B$, has an attractor with eigenvectors $V_1$ and $V_2$ and reactivity $\rho$.
\end{proof}
\par \medskip \noindent 
\textbf{Maximal amplification is bounded.}
In the previous subsection we showed that the instantaneous reactivity, $\rho$, can be arbitrarily large in a reactive attractor, independently of the eigenvalues or (non-orthogonal) eigenvectors. In this subsection we show that when System \eqref{eqn:ODE} has a reactive attractor, the accumulation of that reactive growth is, nevertheless, bounded.
We find upper bounds on the maximal amplification, $\rho_{max}$, in terms of the reactivity radius $\delta_R$ (Theorem \ref{theo:MaxAmpUpperBound}), and in terms of the eigenvector separation radius $\delta_T$ (Corollary \ref{cor:MaxAmpUpperBound-delta_T}). In the next subsection we describe how to calculate $\rho_{max}$ exactly (Theorem \ref{theo:MaxAmpFormula}).
\begin{theo}
\label{theo:MaxAmpUpperBound}
If System \eqref{eqn:ODE} has a reactive attractor at the origin, with reactivity radius 
$\delta_R$, 
then the maximal amplification, $\rho_{max}$ is bounded above by 
\begin{align*}
\rho_{max} < \frac{1}{\cos(2\delta_R)} = -\frac{p}{m_R}.
\end{align*}
\end{theo}
\par
\begin{figure}[t!]
    \centering
    \includegraphics[width=\linewidth]{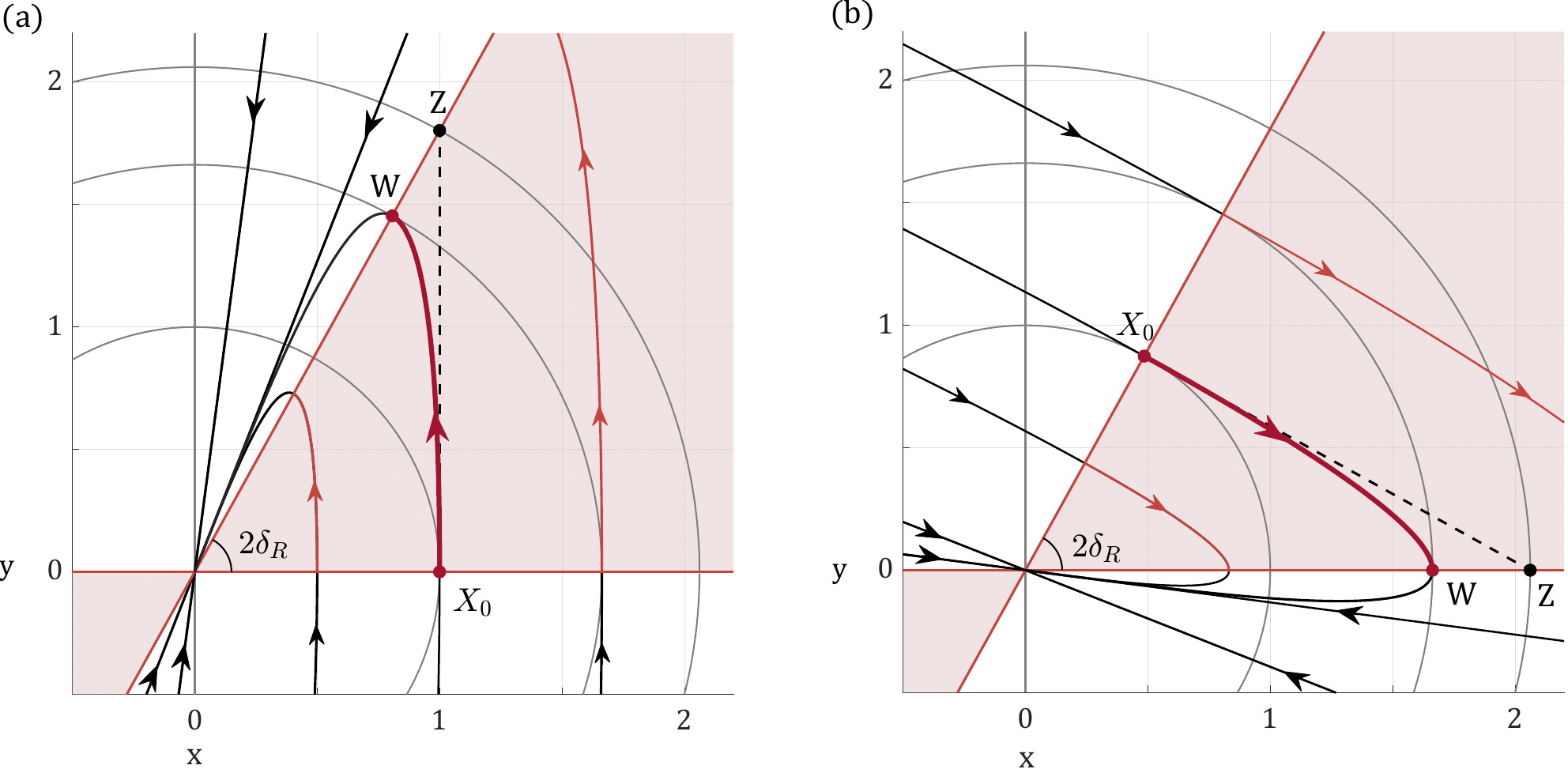}
    \caption{Visual representation of Theorem \ref{theo:MaxAmpUpperBound}. Phase portraits of System \eqref{eqn:ODE} with (a)
    $A=\left(\begin{smallmatrix}
        -1 & -8\\ 0 & -3
    \end{smallmatrix}\right)$
    and (b)
    $B=\left(\begin{smallmatrix}
        -1 & +8\\ 0 & -3
    \end{smallmatrix}\right)$ are plotted in $\Rad$-zeroed form. 
    The coefficient matrices in $\Rad$-zeroed form are given by
    $A_{RZ} \approx \left(\begin{smallmatrix}
        0 & -0.4\\ 7.6 & -4
    \end{smallmatrix}\right)$
    and
    $B_{RZ} \approx \left(\begin{smallmatrix}
        0 & 7.6\\ -0.4 & -4
    \end{smallmatrix}\right)$.
    The two systems have the same eigenvalues and the same reactive radius $\delta_R$, and, being in $\Rad$-zeroed form, they have the same reactive region. The difference between the systems is the sign of $\Tang$, so that trajectories traverse the reactive region counterclockwise in (a), and clockwise in (b).
    \textbf{(a)} The trajectory with initial condition $X_0 = (1,0)^T$ exits the reactive region at the point $W$, which is closer to the origin than the point $Z$. Since $Z$ has $x$-coordinate $1$, $\|Z\|=1/\cos(2\delta_R)$, and so $\rho_{max}=\|W\| \leq  1/\cos(2\delta_R)$. 
    \textbf{(b)} The trajectory with initial condition $X_0=(\cos(2\delta_R),\sin(2\delta_R))^T$ exits the reactive region at the point $W$ on the $x$-axis, and, by similar triangles in (a) and (b), $\rho_{max}=\|W\| \leq \|Z\| = 1/\cos(2\delta_R)$. }
    \label{fig:MaxAmpUpperBound}
\end{figure}
\par 
\begin{proof}
The proof is encapsulated in Figure \ref{fig:MaxAmpUpperBound}.
Without loss of generality, assume that $A$ is in $\Rad$-zeroed form, so, as in Table \ref{tab:std_forms}, the orthovector directions bounding the reactive region
are at $\theta=0$ and $\theta=2\delta_R$, where $\delta_R \in (0,\pi/4)$ since the origin is a reactive attractor. 
The reactive region is given by $\theta \in (0,2\delta_R) \cup (\pi,\pi+2\delta_R)$.  Trajectories experience radial amplification during their passage through the reactive region, and radial attenuation elsewhere.  Since the origin is attracting, we know there is net attenuation, so the maximal amplification is achieved by exactly traversing a single component of the reactive region. 
\par 
Consider the case $\Tang > 0$ on the reactive region (we know $\Tang$ has no zeros on the reactive region since $A$ has no positive eigenvalues). Then the orthovalues, $\mu_1=\Tang(0)$ and $\mu_2=\Tang(2\delta_R)$, are positive and correspond to counter clockwise motion through the reactive region.  To fix ideas, consider the trajectory with initial condition
$X_0= (1,0)^T$ on the ``entrance" boundary of the reactive region. 
This trajectory, $X(t)$, achieves maximal amplification at the point $W$, where it reaches the 
``exit" boundary of the reactive region at 
$\theta=2\delta_R$. Thus $\rho_{max} = \lVert W \rVert$, since $\lVert X_0 \rVert=1$.
\par 
Since $X(t)$ solves System \eqref{eqn:ODE}, which we have assumed is in $\Rad$-zeroed form, the vector field, $dX/dt$, is given by
\begin{align*}
\frac{dX}{dt} = AX= 
\begin{pmatrix} 0 & -\mu_2 \\ \mu_1 & 2m_R\end{pmatrix}
\begin{pmatrix} x \\ y \end{pmatrix}
=
\begin{pmatrix} -\mu_2 y \\ \mu_1 x +2m_Ry \end{pmatrix}.
\end{align*}
Thus, $\frac{dX}{dt}\big|_{X=X_0} = (0,\mu_1)^T$ and points orthogonally into the reactive region (as expected since $X_0$ is an orthovector with orthovalue $\mu_1$).
As $X(t)$ then traverses the reactive region, the first component of $\frac{dX}{dt}$ is negative (since $y>0$ and $\mu_2 >0$). 
Hence, the $x$-coordinate of the point $W$ is less than the $x$-coordinate of $X_0$, and $W$ is closer to the origin along the line $\theta=2\delta_R$ than the point $Z$ where the line $x=1$ intersects $\theta=2\delta_R$.
See Figure \ref{fig:MaxAmpUpperBound}.
 Thus, by trigonometry and Corollary \ref{cor:delta_R_T}, the maximal amplification satisfies:
\begin{align*}
\rho_{max} = \lVert W \rVert < \lVert Z \rVert =  \frac{1}{\cos(2\delta_R)} = -\frac{p}{m_R}.
\end{align*}
\par 
The case when $\Tang < 0$ on the reactive region is proved similarly. In this case trajectories traverse the reactive region clockwise. One approach is to choose an initial condition on the line $\theta=2\delta_R$, as in Figure \ref{fig:MaxAmpUpperBound}b. The calculations are slightly more delicate, but still straightforward. Another approach is to first conjugate with rotation to re-locate the reactive region to $\theta \in (-2\delta_R,0) \cup (\pi-2\delta_R,\pi)$, and then choose $X_0= (1,0)^T$ as before. 
\end{proof}
\par 
\Remark
The upper bound in Theorem \ref{theo:MaxAmpUpperBound} is sharp since the point $W$ can be arbitrarily close to the point $Z$. For example, if $m_T>0$ and $\lambda_1 = 0$, then $\theta_1=2\delta_R \Mod{\pi}$, $\theta_2=\pi/2 \Mod{\pi}$, $W=Z$ and $\rho_{max}=-\frac{p}{m_R}$. Thus, if $\lambda_1$ approaches $0$, $\rho_{max}$ approaches $-\frac{p}{m_R}$. 
Moreover, while the maximal amplification is bounded in terms of $\delta_R$, if $\delta_R$ approaches $\pi/4$, then the sharp upper bound approaches infinity. Thus, linear effects alone can lead to transients of any size.
\par \medskip 
Theorem \ref{theo:MaxAmpUpperBound} gives an upper bound for $\rho_{max}$ based on the reactivity radius, $\delta_R$. In a reactive attractor with real eigenvalues, a bound on the reactivity radius can be determined, in turn, from the eigenvector separation radius, $\delta_T$. In the following corollary of Theorem \ref{theo:MaxAmpUpperBound}, we use this fact to find an alternative (weaker) upper bound on $\rho_{max}$ in terms of $\delta_T$.
\par 
\begin{cor}
\label{cor:MaxAmpUpperBound-delta_T}
If System \eqref{eqn:ODE} has a reactive attractor at the origin, with real eigenvalues, and eigenvector separation  radius $\delta_T \in (0, \pi/2)$, then 
\begin{align*}
\rho_{max} <  \frac{1}{\sin (2\delta_T)}=\frac{p}{p_R}.
\end{align*}
\end{cor}
\begin{proof}
Since System \eqref{eqn:ODE} has real, negative eigenvalues, we know that $\Tang(\theta)\neq 0$ when $\Rad(\theta)>0$, so $\Tang(\theta)$ is single signed on the reactive region. Thus, by splitting into the cases $0<\delta_T< \pi/4$ and $\pi/4<\delta_T <\pi/2$ (corresponding to $m_T<0$ and $m_T>0$, respectively) and considering $\Rad$ and $\Tang$ in $\Tang$-zeroed form, it is straightforward to see that 
\begin{align*}
2\delta_R \in \left(0, \left|\pi /2 - 2\delta_T\right| \right) \subset \left(0,\pi/2 \right).
\end{align*}
Since $\cos(\theta)$ is decreasing on $(0,\pi/2)$, by Theorem \ref{theo:MaxAmpUpperBound} and Corollary \ref{cor:delta_R_T},
\begin{align*}
\rho_{max} < \frac{1}{\cos (2\delta_R)} < 
\frac{1}{\cos (\left|\pi /2 - 2\delta_T\right|)} =
\frac{1}{\sin (2\delta_T)}=\frac{p}{p_R}.
\end{align*}
\end{proof}
\par 
\Remark
    Corollary \ref{cor:MaxAmpUpperBound-delta_T} supports our intuition that, as the eigenvectors in a reactive attractor converge to each other, there is more room for a large reactive region, and the potential for reactivity driven amplification increases.
\par \medskip \noindent
\textbf{Calculating maximal amplification exactly.}
Theorem \ref{theo:MaxAmpUpperBound} and Corollary \ref{cor:MaxAmpUpperBound-delta_T} are useful for providing a quick upper bound on the maximal amplification, $\rho_{max}$, admitted by System \ref{eqn:ODE}. 
In an upcoming paper, we prove Theorem \ref{theo:MaxAmpFormula},
stated below, showing that $\rho_{max}$ can be calculated exactly by integrating $\Rad(\theta(t))$, with respect to time, as a trajectory traverses the reactive region \cite{broda20xxMaxAmp}. The formula for $\rho_{max}$ can be written in a variety of ways, each shedding a different light on what determines the maximal amplification.
\par 
Equation \eqref{eqn:rho-max-lambdas} shows how the balancing act between the transient and asymptotic dynamics of System \eqref{eqn:ODE} that yields the maximal amplification,  $\rho_{max}$, can be captured by an intertwining of the eigenvalues and orthovalues of $A$.  By contrast, Equation \eqref{eqn:rho-max-deltas} shows that, in the case of real eigenvalues, that same balance is captured by the eigenvector and orthovector separations, $\delta_T$ and $\delta_R$.
\begin{theo} 
\label{theo:MaxAmpFormula}
If System \eqref{eqn:ODE} has a reactive attractor at the origin, then the maximal amplification, $\rho_{max}$, admitted by the system is determined by the eigenvalues and orthovalues of $A$ as follows:
\begin{equation}
\label{eqn:rho-max-lambdas}
\rho_{max} = \left( 
\left(
\frac{\lambda_1\mu_2+\lambda_2\mu_1}{\lambda_1\mu_1+\lambda_2\mu_2}
\right)^
{\left(
\frac{\lambda_1+\lambda_2}{\lambda_1-\lambda_2}
\right)}
{\left(
\frac{\mu_1}{\mu_2}
\right)}
\right)^{\frac{1}{2}}
\end{equation}
which can also be written as
\begin{equation}\label{eqn:rho-max-ms}
\rho_{max}  = \left(
\left(\frac{m_Rm_T-p_Rp_T}{m_Rm_T+p_Rp_T}\right)^{
            \left(\frac{m_R}{p_R}\right)}
\left(\frac{m_T+p_T}{m_T-p_T}\right)\right)^{\frac{1}{2}}.
\end{equation}
In the case when $A$ has real eigenvalues, Equation \eqref{eqn:rho-max-lambdas} can also be written in terms of the eigenvector and orthovector separations:
\begin{equation}
\label{eqn:rho-max-deltas}
\rho_{max}  = \left( 
\left(
\frac{\cos(2\delta_R+2\delta_T)}{\cos(2\delta_R-2\delta_T)}
\right)^
{
-\left(
\frac{\cos\left(2\delta_R\right)}{\sin\left(2\delta_T\right)}
\right)
}
{\left(
\frac{\cos(2\delta_T)-\sin(2\delta_R)}{\cos(2\delta_T)+\sin(2\delta_R)}
\right)}
\right)^{\frac{1}{2}}.
\end{equation}
\end{theo}
\par
\begin{Remark}
Since Equation \eqref{eqn:rho-max-deltas} depends only on the zeroes of $\Rad$ and $\Tang$, it reveals that the maximal amplification is independent of the amplitude $p$ of $\Rad$ and $\Tang$. 
Intuitively, in the case of a reactive attractor, even if $p$ is large, making the instantaneous reactivity of System \eqref{eqn:ODE} large, 
this is offset by a correspondingly high speed at which trajectories whip through the reactive region.
Indeed, $p$ controls the time, $t_{max}$, 
taken to reach maximal amplification: If $\Rad$ and $\Tang$ are scaled by $\varepsilon$, then $p$ is scaled by $\varepsilon$, while $\delta_R$ and $\delta_T$ are unchanged. But scaling $\Rad$ and $\Tang$ by $\varepsilon$ is equivalent to scaling $A$ by $\varepsilon$ (Proposition \ref{prop:RadTang-to-A}), which scales the speed of trajectories by $\varepsilon$ and, in turn, scales $t_{max}$ by $1/\varepsilon$. 
\end{Remark}
\par \medskip
Taken together, Theorems \ref{theo:MaxAmpUpperBound} and \ref{theo:MaxAmpFormula} show that linear effects alone can lead to transients of any size, that can last for any length of time. 
But the same theorems also show us how to bound and calculate the size of the transient for a given system.

%% file: RadTang-8-rotating-escape.tex
\section{Surfing the reactivity in nonautonomous systems}\label{sec:nonaut}
In this final section, we illustrate how the framework we have developed can clarify the analysis of a well-known but counter-intuitive effect of transient reactivity. 
Namely, that in a nonautonomous linear system, even if every frozen-time system has a global attractor at the origin, the accumulation of radial growth can be unbounded, yielding an unstable equilibrium at the origin. 
See, for example, \cite{josic2008unstable, lawley2014sensitivity, mierczynski2017instability} and the references therein.
This phenomenon plays an important role when using the variational equation to determine the stability of a periodic orbit of a nonlinear autonomous ODE \cite{guckenheimer1983nonlinear, hirsch2012differential}, or of an equilibrium of a flow-kick system \cite{meyer2018quantifying}.
\par 
We will focus on the family of nonautonomous linear systems generated by conjugating System \eqref{eqn:ODE} with time-dependent rotation, as in the elegant survey by Josi\'{c} and Rosenbaum \cite{josic2008unstable}.
To fix ideas, let $A$ be a $2\times 2$ matrix for which System \eqref{eqn:ODE} has a reactive attractor, and consider the nonautonomous linear system: 
\begin{equation}
\label{eq:non-aut}
    \frac{dX}{dt} = B_k(t)X = M_{kt}^{-1} A M_{kt} X,
\end{equation}
where $M_{kt}$ is the rotation matrix defined in Notation \ref{notn:rotation-matrix}. Note that the conjugation by $M_{kt}$ has the effect of nonautonomously rotating the vector field generated by $A$ clockwise around the origin, with fixed angular velocity $k$.
\par 
The following theorem recasts a result of Josi\'{c} and Rosenbaum \cite[Theorem 3.1]{josic2008unstable} in terms of the orthostructure of $A$.\footnote{To help compare directly with the notation of \cite[Theorem 3.1]{josic2008unstable}, note that the angular velocity $\omega$ in \cite{josic2008unstable} corresponds to $-k$ here (with the sign difference arising from the slight difference in how we conjugate), and the interval $I$  in \cite{josic2008unstable} corresponds to $(\mu_2,\mu_1)$. The orthostructure makes it clear that, when System \eqref{eqn:ODE} has an attractor, $I$ is real and non-empty if and only if the system is reactive. By Theorem \ref{theo:orthostructure}, the value of $D$ in \cite{josic2008unstable} is $m_T$, and the length of $I$ is given by $\mu_1-\mu_2 = 2\sqrt{-\rho_1\rho_2}$.} 
\begin{theorem}
\label{theo:non-aut_result}
Let $A$ be a real $2\times2$ matrix. If the autonomous System \eqref{eqn:ODE} has a reactive attractor with orthovalues $\mu_1 > \mu_2$, then the nonautonomous System \eqref{eq:non-aut} is asymptotically unstable if and only if $-k\in (\mu_2,\mu_1)$.
\end{theorem}
\par 
Before proving Theorem \ref{theo:non-aut_result}, let's build some intuition. For each fixed $s \in \R$, the frozen coefficient matrix $B_k(s)$ is a conjugation by rotation of $A$, and hence shares the same reactivity, $\rho_1$, and reactive radius, $\delta_R$, as $A$ (Corollary \ref{cor:conj-shift}). Thus, at each frozen-time instant in the nonautonomous dynamics there is a reactive region in which solutions move away from the origin. The intuition behind Theorem \ref{theo:non-aut_result} is that if our nonautonomous rate of rotation, $k$, is compatible with the angular velocity of the autonomous System \eqref{eqn:ODE} in the reactive region, then the nonautonomous rotation can exactly balance that angular velocity, thereby trapping solutions in the reactive region forever. These trapped solutions keep `surfing' the reactivity to infinity.
\par 
To cement this intuition, we first prove Lemma \ref{lem:add_RadTang} showing that $\Rad$ and $\Tang$ respect matrix addition. 
Next, to prove Theorem \ref{theo:non-aut_result}, we move from an inertial frame of reference to a co-rotating frame of reference, in which the the nonautonomous system can be represented by the sum of two \textit{autonomous} linear systems: one representing the original matrix $A$, and the other representing the conjugation by rotation. We can then apply Lemma \ref{lem:add_RadTang} and use the autonomous $\Rad$ and $\Tang$ approach to analyze the balance between the angular velocity, $\Tang_A$, within the reactive region of System \eqref{eqn:ODE}, and the nonautonomous rotation rate, $k$, of System \eqref{eq:non-aut}. See Figure \ref{fig:nonaut_escape}.
\begin{figure}[t!]
	\centering
	\includegraphics[width=\linewidth]{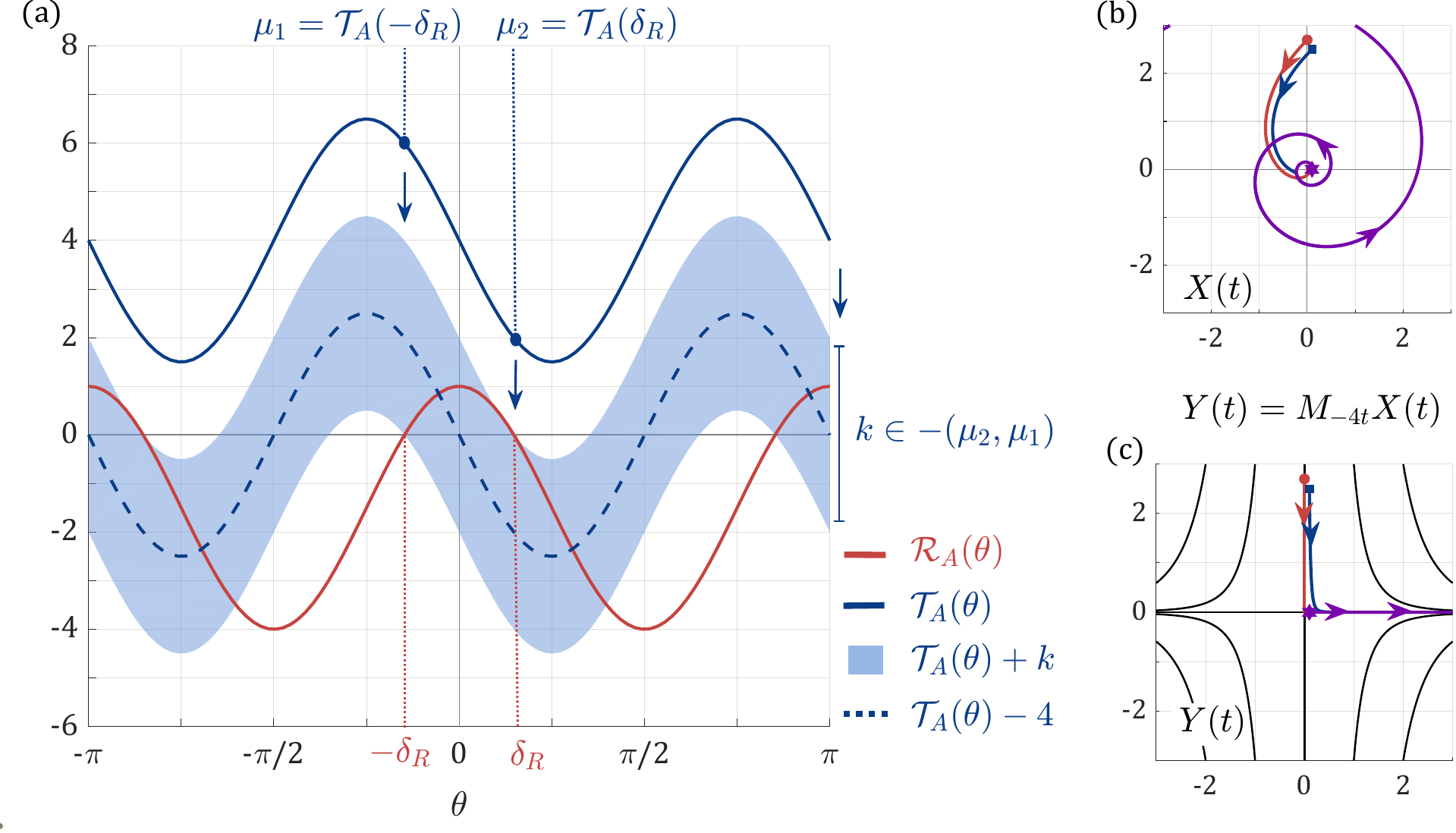}
	\caption{
    Visual representation of Theorem \ref{theo:non-aut_result}, using $A=\left(\begin{smallmatrix} 1 & -4\\ 4 & -4 \end{smallmatrix}\right)$.
		\textbf{(a)} The radial and tangential functions $\Rad_A$ (solid red) and $\Tang_A$ (solid blue) of the autonomous System \eqref{eqn:ODE}. The orthovector angles (zeroes of $\Rad_A$) are at $\pm \delta_R$, and the orthovalues $\mu_1=\Tang_A(-\delta_R)=6$ and $\mu_2=\Tang_A(\delta_R)=2$ are marked on $\Tang_A$. The dashed blue curve shows $\Tang_A$ shifted by $k=-4$.
        If $\Tang_A$ is shifted vertically by any $k\in -(\mu_1,\mu_2) = (-6,-2)$, it will land in the blue shaded region, pushing one of its roots into the reactive set $(-\delta_R, \delta_R)$ where $\Rad>0$. This means $C=A+kJ$ will have a real positive eigenvalue, leading to asymptotic instability (a saddle) in the autonomous System \eqref{eq:co-rot}, and hence in the nonautonomous System \eqref{eq:non-aut}.  
        If $\Tang_A$ is shifted vertically by $k \not \in -(\mu_1,\mu_2)$, then $C=A+kJ$ will not have a positive eigenvalue, and Systems \eqref{eq:co-rot} and \eqref{eq:non-aut} will be stable.
        \textbf{(b)} Three trajectories, $X(t)$, of System \eqref{eq:non-aut} with $k=-4$, in a fixed coordinate frame.
		\textbf{(c)} The corresponding trajectories $Y(t)=M_{-4t}X(t)$ in the co-rotating coordinates of System \eqref{eq:co-rot}, where the unstable saddle behavior is more familiar.}
	\label{fig:nonaut_escape}
\end{figure} 
\begin{lemma}
\label{lem:add_RadTang}
Let $A$ and $B$ be two real $2\times2$ matrices. Then
\begin{align}
\Rad_{A+B}(\theta) &= \Rad_A(\theta) + \Rad_B(\theta),\quad  \mbox{and} \nonumber \\
\Tang_{A+B}(\theta) &= \Tang_A(\theta) + \Tang_B(\theta).\nonumber
\end{align}
\end{lemma}
\begin{proof}
By Theorem \ref{theo:RadTang}, 
\begin{align*}
 \Rad_{A+B}(\theta)X +\Tang_{A+B}(\theta)X^\perp
 & = (A+B)X \\
 & = AX+BX \\
 & = \Rad_A(\theta)X+ \Tang_A(\theta)X^\perp + \Rad_B(\theta)X +\Tang_B(\theta)X^\perp\\
 &= (\Rad_A(\theta) + \Rad_B(\theta))X + (\Tang_A(\theta) + \Tang_B(\theta))X^\perp
 \end{align*}
and hence $\Rad_{A+B}(\theta) = \Rad_A(\theta) + \Rad_B(\theta)$ and $\Tang_{A+B}(\theta) = \Tang_A(\theta) + \Tang_B(\theta)$.
\end{proof} 
\begin{proof}[Proof of Theorem \ref{theo:non-aut_result}]
Since the nonautonomous conjugation rotates 
the vector field clockwise at rate $k$, we move into a co-rotating frame of reference by choosing new coordinates that rotate at the same rate: $Y = M_{kt} X$. Then
\begin{align*}
    \frac{dY}{dt} &= \frac{d}{dt}\big[M_{kt}X\big] 
    = M_{kt}\frac{dX}{dt} + \frac{d}{dt}\big[M_{kt}\big]X \\
    &=AM_{kt}X + 
    \begin{pmatrix}
        0 & -k\\k&0
    \end{pmatrix} M_{kt}X\\
    &= (A + kJ )Y. 
 \end{align*}
Thus, in the co-rotating coordinate frame the nonautonomous system is represented by an \textit{autonomous} linear system:
 \begin{align}
   \frac{dY}{dt} = CY, \ \mbox{ where }  C=A+kJ. \label{eq:co-rot}
\end{align}
By Lemma \ref{lem:add_RadTang}, the radial and tangential functions of System \eqref{eq:co-rot} 
are given by 
\begin{align}
\nonumber
\Rad_C(\theta) & = \Rad_A(\theta) + \Rad_{kJ}(\theta) = \Rad_A(\theta), \ \text{and} \\
\Tang_C(\theta) & = \Tang_A(\theta) + \Tang_{kJ}(\theta) = \Tang_A(\theta) + k,
\label{eqn:T_C}
\end{align}
since $\Rad_{kJ}(\theta) \equiv 0$ and $\Tang_{kJ}(\theta) \equiv k$. Thus, in the new co-rotating coordinates, the  nonautonomous rotation imposed by conjugation by $M_{kt}$ in System \eqref{eq:non-aut} manifests as a vertical translation of $\Tang_A$. See Figure \ref{fig:nonaut_escape}a.
\par
The asymptotic radial dynamics of the nonautonomous System \eqref{eq:non-aut} are equivalent to those of the autonomous System 
\eqref{eq:co-rot} via the change of coordinates $X(t) = M_{kt}^{-1}Y(t)$ since $\|X(t)\| = \|Y(t)\|$. 
Thus, to complete the proof, we can simply determine when the autonomous dynamics of System \eqref{eq:co-rot} are asymptotically unstable.
\par
Since $\Rad_C=\Rad_A$, System \eqref{eq:co-rot} has the same reactive set $\Rset$, midline $m_R<0$, and reactivity $\rho>0$ as System \eqref{eqn:ODE}. Thus, under System \eqref{eq:co-rot}, the origin is generically either a reactive attractor or a saddle (and therefore unstable), depending on $k$.
In particular, it is a saddle if and only if $C$ has a positive real eigenvalue.
\par
Let $\lambda_1$ denote the largest eigenvalue of $C$, and $\theta_1$ denote the corresponding eigenvector angle. Then,  applying Theorem \ref{theo:eigenstructure}, we have
\begin{align*}
\lambda_1>0 & \iff \theta_1 \in \Rset \\
    & \iff \Tang_C(\theta) = 0 \mbox{ for some } \theta \in \Rset\\
    & \iff \Tang_A(\theta) = -k \mbox{ for some } \theta \in \Rset,  \mbox{ by Equation \eqref{eqn:T_C}}\\
    & \iff -k \in \Tang_A(\Rset)\\
    & \iff -k \in (\mu_2,\mu_1), \mbox{ by Theorem \ref{theo:orthostructure} applied to $A$.}
\end{align*}
\end{proof}
\par
The example illustrated in Figure \ref{fig:nonaut_escape}b shows a saddle of the nonautonomously rotating System \eqref{eq:non-aut} in a fixed coordinate frame, while \ref{fig:nonaut_escape}c shows the same saddle in the co-rotating coordinates, looking more familiar since it solves the autonomous System \eqref{eq:co-rot}. 

%% file: RadTang-9-Acknowledgements.tex
\par \medskip
{\bf Acknowledgement.}
We thank Margaret Beck for helpful conversations. We are also grateful for support from Lenfest Grants (Washington and Lee University, JB), Clare Booth Luce Program (AH-H) the Institute for Computational Sustainability (NSF Award CCF-1522054), the Center for Analysis and Prediction of Pandemic Expansion (NSF Award DBI-2412115), 
and the American Institute of Mathematics.